\documentclass[11pt,reqno]{amsart}
\usepackage{color, amsmath,amssymb, amsfonts, amstext,amsthm, latexsym}
\allowdisplaybreaks
\newcommand{\RR}{{\mathbb R}}

\newcommand{\EE}{{\mathbb E}}

\newcommand{\LL}{{\mathbb L}}
  
\newcommand{\PP}{{\mathbb P}}

\numberwithin{equation}{section}
\newtheorem{theorem}{Theorem}[section]
\newtheorem{defn}[theorem]{Definition}

\newtheorem{lemma}[theorem]{Lemma}
\newtheorem{remark}[theorem]{Remark}
\newtheorem{prop}[theorem]{Proposition}

\begin{document}
\title[Strong Convergence for 2D NS  time  numerical schemes]
{Strong $L^2$ convergence of time numerical schemes  
 \\ for the stochastic 2D Navier-Stokes equations}

\author[H. Bessaih]{Hakima Bessaih}
\address{University of Wyoming, Department of Mathematics and Statistics, Dept. 3036, 1000
East University Avenue, Laramie WY 82071, United States}
\email{ bessaih@uwyo.edu}

\author[A. Millet]{ Annie Millet}
\address{SAMM, EA 4543,
Universit\'e Paris 1 Panth\'eon Sorbonne, 90 Rue de
Tolbiac, 75634 Paris Cedex France {\it and} Laboratoire de
Probabilit\'es, Statistique et Mod\'elisation, UMR 8001, 
  Universit\'es Paris~6-Paris~7} 
\email{amillet@univ-paris1.fr}

\thanks{  Hakima Bessaih was partially supported by NSF grant DMS 1418838. }  



\begin{abstract}
We prove that some time  discretization schemes for the 2D Navier-Stokes equations on the torus 
subject to a random perturbation converge in $L^2(\Omega)$.
This refines previous results which only established the convergence in probability of these numerical approximations. Using 
exponential  moment estimates of
the solution of the stochastic Navier-Stokes equations and convergence of a   localized scheme,  we can prove strong convergence of 
 fully implicit and semi-implicit  time Euler discretizations, and of a splitting scheme. 
The  speed of the $L^2(\Omega)$-convergence depends on the diffusion coefficient and on the viscosity parameter. 
\end{abstract}

\maketitle

\section{Introduction}\label{s1} \smallskip
An incompressible fluid flow dynamic is described by the so-called incompressible Navier-Stokes equations. 
The fluid flow is defined by a velocity field and a pressure term that evolve in a very particular way. 
These equations are parametrized by the viscosity coefficient $\nu>0$. 
Their quantitative and qualitative properties depend on the dimensional setting. 
For example, while the well posedness of global weak solutions of the 2D Navier-Stokes   is  well known and established, 
 the uniqueness of global weak solutions for the 3D case is completely open.
In this paper, we will focus on the 2D incompressible Navier-Stokes equations in a bounded domain $D= [0,L]^2$,
subject to an external forcing defined as:
\begin{align} \label{2D-NS}
 \partial_t u - \nu \Delta u + (u\cdot \nabla) u + \nabla \pi & = G(u) dW\quad \mbox{\rm in } \quad (0,T)\times D,\\
 \mbox{\rm div }u&=0 \quad \mbox{\rm in } \quad (0,T)\times D,
 \end{align}
 where  $T>0$. The process 
$u: \Omega\times (0,T)\times D  \to \RR^2$  is  the velocity field 
 with initial condition $u_0$ in $D$ and periodic boundary conditions $u(t,x+L v_i)=u(t,x)$ on  $(0,T)\times \partial D$, 
 where $v_i$, $i=1,2$
 denotes the canonical basis of $\RR^2$, and $\pi : \Omega\times (0,T)\times D  \to \RR$  is  the  pressure. 
 
 The external force is described by a stochastic perturbation and will be defined in detail later. 
 Here $G$ is a diffusion coefficient with global Lipschitz conditions.
Let $(\Omega, {\mathcal F}, ({\mathcal F}_t),  \PP)$ denote a filtered probability space and $W$ be a Wiener process to be precisely defined later. 

There is an extensive  literature concerning the deterministic models and we refer to the books of Temam; see \cite{Tem, Temam-84} for known 
results. The stochastic case has also been widely investigated, see \cite{FG} for some very general results and the references therein. 
For the 2D case, unique global weak and strong solutions (in the PDE sense) are constructed for both additive and multiplicative noise, 
and without being exhaustive, we refer to \cite{Breckner,  ChuMil}.  

Numerical schemes and algorithms have been introduced to best approximate and construct solutions for PDEs. 
A similar approach has started to emerge for stochastic models and in particular SPDEs and has known a strong
 interest by the probability community.  
Many algorithms  based on either finite difference, finite element  or spectral Galerkin methods 
(for the space discretization), and on either Euler schemes, Crank-Nicolson or Runge-Kutta schemes (for the temporal discretization) 
have been introduced for both the linear and nonlinear cases. Their rates of convergence have been widely investigated.
The literature on numerical analysis for SPDEs is now very extensive. 
When the models are either linear, have global Lipschitz properties or more generally some monotonicity property, 
then there is extensive literature, see \cite{Ben1, Ben2}. Moreover, in this case the convergence is proven to be in mean square. 
When nonlinearities are involved that are not of Lipschitz or monotone type,  then a rate of  convergence 
 in mean square is difficult to  obtain.  
  Indeed, because of the stochastic perturbation, there is no way of using the Gronwall lemma after taking the expectation of the
   error bound   
  because it involves a nonlinear term that is usually in a quadratic form. 
  One way of getting around it is to localize the nonlinear term in order to get a linear inequality and then use the Gronwall lemma. 
  This gives rise to  a rate of   convergence in probability, that was first introduced by J.~Printems \cite{Pri}.

The stochastic Navier-Stokes equations with a multiplicative noise \eqref{2D-NS}
 have been investigated  by Z.~Brzezniak, E.~Carelli and A.~Prohl in \cite{BrCaPr}. 
 There,   space discretization based on finite elements  and Euler schemes for the time discretization  have been implemented. 
 The numerical scheme was proven to converge in probability  with a particular rate. A similar  problem 
  has been investigated by E.~Carelli and A.~Prohl
 in \cite{CarPro}, 
 with more focus on various Euler schemes including semi-implicit and fully implicit ones. This  gave rise to a slightly different rate of convergence, 
 although still in probability.  Again, the main tool used is the localization of the nonlinear term over a probability space of 
  "large"  probability.  
In \cite{BeBrMi}, the authors used a splitting method,  based on the  Lie-Trotter  formula,
 proving again some rate of convergence   in probability   of the 
numerical scheme.  In \cite{Dor}, P. D\"orsek studied 
a semigroup splitting and used cubature approximations,  obtaining interesting results for an additive
 noise.  When the noise is  additive,  a pathwise argument  was used by H.~Breckner in \cite{Breckner}; 
 convergence almost sure 
 and in mean  was obtained, although no rate of convergence was explicitly given. 
 To  the best of our knowledge there is no result about  a strong speed of convergence
 for the stochastic Navier-Stokes equations in the current literature.

 Numerical schemes for stochastic nonlinear models with local Lipschitz nonlinearities  
 related with the Navier-Stokes equations have been studied by several authors.
 For nonlinear parabolic SPDEs,  in \cite{ BJ} D.~Bl\"omker and A.~Jentzen proved a speed of convergence in probability
 of Galerkin approximations of the stochastic Burgers equation, which is a simpler nonlinear PDE  which has some   similarity 
 with  the   Navier-Stokes equation.   In \cite{BHRR},  an abstract stochastic nonlinear evolution equation in a separable Hilbert space was 
  investigated, including the GOY and Sabra shell models. 
    These adimensional 
   models are phenomenological approximations of the Navier-Stokes equations. The authors  proved the convergence 
  in probability in a fractional Sobolev space $H^s$, $0\leq s <\frac{1}{4}$,  of  
   a space-time numerical scheme defined in terms of a Galerkin approximation in space, and a semi-implicit Euler-Maruyama scheme in time. 
  For the Burgers  as well as more general nonlinear SPDEs subject to space-time white noise driven
   perturbation, A.~Jentzen, D.~Salinova and T.~Welti  proved in \cite{JeSaWe} 
 the strong convergence of the scheme but did not give a rate of convergence.  
 \medskip
 
 In this paper, we focus on the stochastic 2D Navier-Stokes equations and  would like to go one step further,  
that is,   obtain a strong  speed of  convergence in mean square instead of the convergence in probability. 
In fact, the main goal is twofold. On one hand, we will improve the convergence  from convergence in probability to  
$L^2(\Omega)$- convergence, the so-called strong convergence in mean square. On the other hand,
 we will also improve the rate of convergence from logarithmic to  almost polynomial.
 
  To explain the method, the paper will deal with two different algorithms: the splitting scheme used in \cite{BeBrMi} 
  and the implicit Euler schemes used in \cite{CarPro}. 
   In the case of a diffusion coefficient $G$ with linear growth conditions, which may depend on the solution and its gradient  for the
   Euler schemes,
  we prove that the speed of convergence of both schemes is any negative power of the logarithm of the time mesh $\frac{T}{N}$ when the
  initial condition  belongs  to $ {\mathbb W}^{1,2}$  and is divergence free. 
  In the case of an additive noise - or under a slight generalization of such
  a noise - we prove that the strong $L^2(\Omega)$ speed of convergence of the fully or semi implicit Euler schemes introduced by Carelli and Prohl 
  in \cite{CarPro} is polynomial in the time mesh. This speed depends on the viscosity coefficient $\nu$ and on the length of the time interval $T$.
  When $T$ is small,  or when  $\nu$ is large, this speed is close to the best one which can be achieved in time,
   that is almost $\frac{1}{4}$.
   This is consistent with the time regularity of the strong solution to the stochastic Navier-Stokes equations,  due
  to the scaling between the time and space variables  in the heat kernel, and to the stochastic integral. 
  
  Let us try to explain the steps of our method here before going into  more details later on in the paper. 
   As we explained earlier, the main difficulty that prevents getting the strong convergence in mean square is due
    to the nonlinear term $(u\cdot \nabla) u$. Indeed, in order to bound  the error 
    $e_{k}:=u(t_{k})-u_N(t_{k})$ over the grid points $t_{k},\ k=1,\dots, N$, in some implicit Euler method, one has to upper bound 
$$\EE \|(u(t_{k})\cdot \nabla) u(t_{k})- (u_N(t_{k})\cdot \nabla) u_N(t_{k})\|_{V'}.$$
To close the estimates and use some Gronwall lemma, the tool  used in \cite{CarPro} (as well as in  \cite{BeBrMi})
is to localize on a subspace of $\Omega$. However, in both previous results, the localization set was depending on the
discretization. In this work, we make  slightly different computations, based on the antisymmetry of the bilinear term, and localize on
sets which only depend on the solution to the stochastic Navier Stokes equations \eqref{2D-NS}, such as 
 $\Omega_{N}^{M}$ defined by \eqref{Omegak} for the Euler schemes. 
Hence, one obtains for example 
\begin{equation*}			
\EE\Big( 1_{\Omega_{N}^M } \max_{1\leq k \leq N}   |e_k|_{\LL^2}^2  \Big) \leq C \; \exp\big[C_1(M) T\big] \;  \Big(\frac{T}{N}\Big)^\eta , 
\end{equation*}
where $\eta <\frac{1}{2}$ and  $C_1(M)$ is a constant depending on the bound $M$ of the $\LL^2$-norm of $\nabla u$
 imposed on the localization set $\Omega_{N}^{M}$.  The exponent $\eta$ is
natural and related to the time regularity of the solution $u$ when the initial condition $u_0$ belongs to ${\mathbb W}^{1,2}$. 

In order to prove the strong speed of convergence,  we will use the partition of $\Omega$ into $\Omega_{N}^M$ and its
complement for some threshold $M$ depending on $N$. More precisely, 
we have to balance the upper estimate of the  moments localized on the set $\Omega_N^{M(N)}$
for some well chosen sequence $M(N)$, going to infinity as $N$ does, and a similar upper estimate of the $L^2(\Omega)$ moment of the error
localized on the complement of the set $\Omega_{N}^{M(N)}$.  
 This is performed by upper estimating moments of $u(t_k)$ and of $u_N(t_k)$
uniformly in $N$ and $k$, estimating $\PP\big( (\Omega_N^{M(N)})^c\big)$, and using the H\"older inequality
\[ \EE\Big( 1_{(\Omega_N^{M(N)})^c} \max_{1\leq k \leq N}   |e_k|_{\LL^2}^2  \Big) \leq \Big( \PP\big((\Omega_N^{M(N)})^c \big) \Big)^{\frac{1}{p}}
 \Big[ \EE \Big( \sup_{0\leq s\leq T} |u(s)|_{\LL^2}^{2q} + \max_{0\leq k\leq N} |u_N(t_k)|_{\LL^2}^{2q} \Big) \Big]^{\frac{1}{q}}, \] 
where $p$ and $q$ are conjugate exponents.  A similar bound was already used in \cite{HutJen} in a different numerical
framework. 
The upper estimate of the probability of the "bad" set depends on the assumptions on the diffusion coefficient. Note that since we are localizing
on a set which does not depend on the discretization scheme, only moments of the solution to the stochastic Navier Stokes equation 
\eqref{2D-NS} have to
be dealt with.

 In the case of globally Lipschitz coefficient $G$, we use bounds of  various moments of $u$ in    ${\mathbb W}^{1,2}$. 
For both schemes,   the strong speed of convergence is again in the logarithmic scale; when the initial condition is deterministic,
 it is any negative power of $\ln(N)$.  
 
 In the case of an additive noise, 
 we use  a slight extension of exponential moments of the solution of \eqref{2D-NS} in 
 vorticity formulation proved previously by M.~Hairer and J.~Mattingly  in \cite{HaiMat}, given by 
$\EE\big(\sup_{t\in [0,T]} \exp(\alpha_0 |\nabla u(t)|_{\LL^2}^2) \big) <\infty$ for some $\alpha_0 >0$. 
 This yields a better speed of convergence,  due to that fact that  the polynomial Markov inequality is replaced by an exponential one.

For the implicit Euler scheme the strong  speed of convergence 
is polynomial with exponent $\gamma< \frac{1}{2}$  that depends on the viscosity $\nu$. 
Note that for large $\nu$, $\gamma$ approaches $\frac{1}{2}$. 
 For the splitting scheme, 
the strong speed of convergence we obtain in this paper is better than that of the convergence in probability
proven in \cite{BeBrMi}, although not polynomial; it is of the form $c \exp(-C \sqrt{N})$. 

In this paper, 
we  only deal with  time discretization, unlike in \cite{CarPro} where a space-time discretization is studied.
Furthermore, in order to keep the paper in a reasonable size and present simple arguments to follow by the reader,
  we assume  that $G$ does not depend on time. We add relevant comments and remarks on the assumptions to be
added in case of time dependent coefficients   for the implicit Euler schemes  (see section \ref{sec_time_dependent}).  
 \smallskip

The paper is organized as follows. Section \ref{preliminary} recalls basic properties of the 2D Navier Stokes
equations, functional spaces and  strong solutions.  We formulate the  assumptions on  the noise.
 The splitting scheme  is described in Section \ref{sec_splitting} and various moments estimates previously used in  
 \cite{BeBrMi} are recalled. 
 The strategy for proving the strong speed of convergence is described and explained in details. 
The same strategy is used for the Euler schemes   in    Section \ref{sec_Euler}, which  is devoted to the fully implicit and 
semi implicit Euler schemes 
studied previously in \cite{CarPro}. Their strong speed of convergence is proved with a rate of convergence that  is  polynomial  when the 
exponential moment is used.    Finally, Section \ref{Appendix} provides on one hand an improved  moment estimate
for an auxiliary process used in the splitting scheme, and on the other hand the proof for the exponential moment estimates of the gradient of
the solution to \eqref{2D-NS} .

As usual, except if specified otherwise, $C$ denotes a positive constant that may change  throughout the paper,
 and $C(a)$ denotes  a positive constant depending on the parameter $a$.

\section{Notations and preliminary results}\label{preliminary} 
Let ${\mathbb L}^p:=L^p(D)^2$ (resp. ${\mathbb W}^{k,p}:=W^{k,p}(D)^2$)  denote the usual Lebesgue and Sobolev spaces of 
vector-valued functions
endowed with the norms $|\cdot |_{\LL^2}$ (resp. $\|\cdot \|_{{\mathbb W}^{k,p}}$). 
 In what follows, we will  consider velocity fields that have   mean zero  
 over  $[0,L]^2$.  Let $\LL^2_{per}$ denote the subset of $\LL^2$  periodic functions with mean zero over   $[0,L]^2$, and let
\begin{align*}
  H:= &\{ u\in \LL^2_{per} \; : \; {\rm div }\;  u=0 \quad \mbox {\rm weakly in }\;  D \}, \qquad
  V:=  H \cap {\mathbb W}^{1,2}
  \end{align*}
   be  separable  Hilbert spaces.  The space $H$ inherits its inner product  denoted by $(\cdot,\cdot)$ and its norm from $\LL^2$.
   The norm in $V$, inherited from ${\mathbb W}^{1,2}$, is denoted by $\| \cdot \|_V$.   Moreover,   let  $V'$ be the dual space of $V$ 
   with respect to the Gelfand triple,  
  $\langle\cdot,\cdot\rangle$ denotes the duality between $V'$ and $V$.    Let $A=- \Delta$ with its domain  
  $\mbox{\rm Dom}(A)={\mathbb W}^{2,2}\cap H$. 
  
  Let $b:V^3 \to \RR$ denote the trilinear map defined by 
  \[ b(u_1,u_2,u_3):=\int_D  \big(u_1(x)\cdot \nabla u_2(x)\big)\cdot u_3(x)\, dx, \]
  which by the incompressibility condition satisfies  $b(u_1,u_2, u_3)=-b(u_1,u_3,u_2)$  for $u_i \in V$, $i=1,2,3$. 
  There exists a continuous bilinear map $B:V\times V \mapsto
  V'$ such that
  \[ \langle B(u_1,u_2), u_3\rangle = b(u_1,u_2,u_3), \quad \mbox{\rm for all } \; u_i\in V, \; i=1,2,3.\]
  The map  $B$ satisfies the following antisymmetry relations:
  \begin{equation} \label{B}
  \langle B(u_1,u_2), u_3\rangle = - \langle B(u_1,u_3), u_2\rangle , \quad \langle B(u_1,u_2),  u_2\rangle = 0 \qquad \mbox {\rm for all } \quad u_i\in V.
  \end{equation}
  Furthermore, the Gagliardo-Nirenberg inequality implies that for $X:=H\cap \LL^4(D)$ we have 
  \begin{equation} \label{interpol}
  \|u\|_X^2 \leq \bar{C} \; |u|_{\LL^2} \, |\nabla u|_{\LL^2} \leq \frac{\bar{C}}{2} \|u\|_V^2 
  \end{equation}
  for some positive constant $\bar{C}$. 
  For $u\in V$ set $B(u):=B(u,u)$ and recall some well-known properties of $B$, which easily follow from the H\"older and Young inequalities:
   given any $\beta >0$   we have 
  \begin{align}
 & | \langle B(u_1,u_2), u_3\rangle | \leq \beta \|u_3\|_V^2 + \frac{1}{4 \beta} \|u_1\|_X\, \|u_2\|_X, \label{majB-X}\\
  & |\langle B(u_1) - B(u_2) \, , u_1-u_2\rangle | \leq \beta \|u_1-u_2\|_V^2 + C_\beta |u_1-u_2|_{\LL^2}^2 \|u_1\|_X^4, \label{B-B}
  \end{align}
  for  $u_i\in V$, $i=1,2,3$, where 
  \begin{equation} \label{def_Ceta}
  C_\beta = \frac{\bar{C}^2 3^3}{4^4 \beta^3},
  \end{equation}
  and $\bar{C}$ is defined by \eqref{interpol}. 
  
  Let $K$ be a separable Hilbert space and $(W(t), t\in [0,T])$ be a $K$-cylindrical Wiener process defined on the probability
   space $(\Omega, {\mathcal F},
  ({\mathcal F}_t),  \PP)$. 
  For technical reasons, we assume that the initial condition $u_0$ belongs to $L^p(\Omega ; V)$ for some $p\in [2,\infty]$, 
   and  only consider {\it strong solutions}  in the PDE sense. 
 Given two Hilbert spaces $H_1$ and $H_2$, let   ${\mathcal L}_2(H_1,H_2)$ denote the set of Hilbert-Schmidt operators from $H_1$ to $H_2$. 
  The diffusion coefficient $G$ satisfies the following assumption:\\
 { \bf Condition (G1)} Assume that 
  $G : V \to {\mathcal L}_2(K,H)$ is continuous and there exist positive constants $K_i$, $i=0,1$ 
  and $L_1$  
  such
  that for $u,v\in  V  $: 
  \begin{align}
\| G(u)\|_{{\mathcal L}_2(K,H)}^2 &\leq K_0 + K_1 |u|_{\LL^2}^2, \label{growthG_H}\\  
 \|G(u)-G(v)\|_{{\mathcal L}_2(K,H)}^2 &\leq L_1 |u-v|_{\LL^2}^2 . 
 \label{LipG}
  \end{align}

  Finally, note that the following identity involving the Stokes operator $A$ and the bilinear term holds (see e.g. \cite{Tem} Lemma 3.1):
  \begin{equation}   \label{A-B}
  \langle B(u), Au \rangle =0, \quad u\in \mbox{\rm Dom}(A). 
  \end{equation}
  We also suppose that $G$ satisfies the following assumptions:  \\

 {\bf Condition (G2)} The coefficient  $G:
 {\rm Dom}(A) \to {\mathcal L}_2(K,V)$ and there exist positive constants ${K}_i$, $i=0,1$, 
 and $L_1$ 
 such that for every 
 $u,v\in \mbox{\rm Dom}(A)$: 
 \begin{align}
 \| G(u)\|_{{\mathcal L}_2(K,V)}^2 &\le   {K_0} + K_1 \|u\|_V^2,  \label{growthG_V}\\ 
 \| G(u)-  G(v) |_{{\mathcal L}_2(K,V)}^2 & \le L_1 \|u-v\|_V^2 .\label{LipG_V}  
 \end{align}

 We define a  strong solution of \eqref{2D-NS} as follows (see Definition 2.1 in \cite{CarPro}): 
 \begin{defn}
 We say that equation \eqref{2D-NS} has a strong  solution if:
 \begin{itemize}
 \item  $u $ is an adapted $V$-valued process,
 \item $\PP$ a.s. we have $u\in C([0,T];V) \cap L^2(0,T; \mbox{\rm Dom}(A))$,
 \item  $\PP\;  \mbox{\rm a.s.}$
  \begin{align*}
  \big(u(t), \phi\big) +& \nu \int_0^t \big( \nabla u(s), \nabla \phi\big) ds + \int_0^t \big\langle [u(s) \cdot \nabla]u(s), \phi\big\rangle ds \\
& =
 \big( u_0, \phi) + \int_0^t \big( \phi ,  G(u(s)) dW(s) \big)
 \end{align*}
for every $t\in [0,T]$ and every $\phi \in V$.
  \end{itemize}
 \end{defn}
 
 As usual, by projecting  \eqref{2D-NS} on divergence free fields, the pressure term is implicitly in the space $V$ and can be recovered afterwards. 
 Proposition 2.2 in \cite{BesMil} (see also \cite{BeBrMi}, Theorem 4.1) shows the following:
  \begin{theorem} \label{strong_wp}
  Assume that $u_0$ is a $V$-valued, ${\mathcal F}_0$-measurable  random variable such that $\EE \big( \|u_0\|_V^{2p}\big) <\infty$ 
  for some real number $p\in [2,\infty)$.  Assume that the conditions {\bf (G1)} and {\bf (G2)} are satisfied.  
  Then  there exists a unique    solution $u$ to equation \eqref{2D-NS}. 
  Furthermore,  for some positive constant $C$ we have
  \begin{equation}   \label{bound_u}
  \EE\Big( \sup_{t\in [0,T]} \|u(t)\|_V^{2p} + \int_0^T |Au(s)|_{\LL^2}^2 \big( 1+\|u(s)\|_V^{2(p-1)}\big) ds \Big) \leq C\big[ 1+ \EE (\|u_0\|_V^{2p}) \big].
  \end{equation}
  \end{theorem}

\section{Time splitting scheme} 		\label{sec_splitting}
In this section, we prove the strong $L^2(\Omega)$ convergence of the splitting scheme introduced in \cite{BeBrMi}.
\subsection{Description of the splitting scheme} 
 Let $N>1$, $h=\frac{T}{N}$ denote the time mesh, and  $t_i=\frac{iT}{N}$, 
$i=0, \cdots, N$ denote a partition of the time interval $[0,T]$. Let  $F: V\to V'$  be defined by
\begin{equation}			\label{def_F}
F(u) = \nu Au + B(u,u).
\end{equation}


\noindent   Note that the formulation of {\bf (G2)} is slighty different from that used in \cite{BeBrMi}. 
They are equivalent due to the inequality
  $|\nabla u|_{\LL^2} \leq C | \mbox{\rm curl } u|_{\LL^2}$ for $u\in V$, where $\mbox{\rm curl } u = \partial_{x_1} u_2 - \partial_{x_2} u_1$.

Set $t_{-1}=-\frac{T}{N}$. For $t\in [t_{-1},0)$ set $y^N(t)=u^N(t)=u_0$ and ${\mathcal F}_t={\mathcal F}_0$. The  approximation. 
 $(y^N,u^N)$ is defined
by induction as follows. Suppose that the processes $u^N(t)$ and $y^N(t)$ are defined for $t\in [t_{i-1},t_i)$ and that $y^N(t_i^-)$ is
$H$-valued and ${\mathcal F}_{t_i}$-measurable. Then for $t\in [t_i,t_{i+1})$, 
 $u^N(t)$ with initial condition $y^N(t_i^-)$ at time $t_i$, is the unique solution of equation:
\begin{equation}			\label{def_uN}
\frac{d}{dt} u^N(t) + F( u^N(t)) = 0, \quad t\in [t_i,t_{i+1}), \quad  u^N(t_i)=u^N(t_i^+) = y^N(t_i^-). 
\end{equation}
Then $u^N(t_{i+1}^-)$ is well-defined,   $H$-valued and ${\mathcal F}_{t_i}$-measurable.  Then  set
$y^N(t_i) = u^N(t_{i+1}^-)$,  and   for $t\in [t_i,t_{i+1})$ 
 define $y^N(t)$ 
as the unique solution of equation  
\begin{equation}			\label{def_yN}
d y^N(t) = G(y^N(t)) dW(t), \quad t\in [t_i,t_{i+1}), \quad  y^N(t_i)=y^N(t_i^+) = u^N(t_{i+1}^-). 
\end{equation}
Finally, set $u^N(T)=y^N(T)=y^N(T^-)$.  The processes $u^N$ and $y^N$ are well-defined and have finite moments
as proved in \cite{BeBrMi}, Lemma~4.2.
\begin{theorem}			\label{moments_spscheme}
Let 
$u_0$ be a $V$-valued, ${\mathcal F}_0$ random variable such that $\EE(\|u_0\|_V^{2p}) <\infty$ for some real number $p\geq 2$.
Suppose that $G$ satisfies conditions  {\bf (G1)} and {\bf (G2)}.  
 Then there exists a positive constant $C$ such that for every
integer $N\geq 1$
\begin{align}				\label{moments_splitting} 
\sup_{t\in [0,T]} \EE\big( \|u^N(t)\|_V^{2p} + \|y^N(t)\|_V^{2p}\big) + \EE\int_0^T \big(1+\|u^N(t)\|_V^{2(p-1)}\big) |Au^N(t)|_{\LL^2}^2 dt \leq C. 
\end{align}
\end{theorem}
The following result gives a bound of the difference between $u^N$ and $y^N$ (see \cite{BeBrMi}, Proposition 4.3).
\begin{prop}		\label{y-u}
Let $u_0$ be ${\mathcal F}_0$-measurable such that $\EE(\|u_0\|_V^4)<\infty$ and $G$ satisfies
conditions  {\bf (G1)} and {\bf (G2)}.  
Then there exists a positive constant $C:=C(T)$ such that for every integer $N\geq 1$
\begin{equation} 		\label{yN-uN}
\EE\int_0^T \|y^N(t) - u^N(t)\|_V^2 dt \leq \frac{C}{N}. 
\end{equation}
\end{prop}
For technical reasons, let us consider the process $(z^N(t), t\in [0,T])$ which mixes $u^N$ and $y^N$, and is defined by
\begin{equation}		\label{def_zN}
z^N(t)=u_0 - \int_0^t F(u^N(s)) ds + \int_0^t G(y^N(s)) dW(s).
\end{equation}
The process $z^N$ is a.s. continuous on $[0,T]$. Note that $z^N(t_k)=y^N(t_k^-) = u^N(t_k)$ for $k=0, 1, \cdots, N$. 
The following result gives bounds of  the $V$-norm of $z^N-u^N$ and $z^N-y^N$. It is an extension of Lemma 4.4 in \cite{BeBrMi}. 
\begin{prop}			\label{prop_u-z}
Let $u_0$ be $V$-valued, ${\mathcal F}_0$-measurable such that $\EE(\|u_0\|_V^{2p})<\infty$ for some integer $p\geq 2$
and let $G$ satisfy conditions   {\bf (G1)}  and {\bf (G2)}. 
Then there exists a constant $C:=C(p,T)$ such that for every integer $N\geq 1$
\begin{equation}			\label{zN-uN_V}
\EE\Big( \sup_{t\in [0,T]}  \|z^N(t)-u^N(t)\|_V^{2p} \Big) \leq \frac{C}{N^{p-1}}  \; \;  \mbox{\rm and } \;  
\sup_{t\in [0,T]} \EE \big( \|z^N(t)-u^N(t)\|_V^{2p} \big) \leq \frac{C}{N^p}. 
\end{equation}
\end{prop}
Note that combining \eqref{yN-uN} and \eqref{zN-uN_V}, we deduce that if $\EE(\|u_0\|_V^4)<\infty$, 
\begin{equation} 			\label{zN-yN}
\EE\int_0^T \! \|z^N(t)-y^N(t)\|_V^2 dt \leq \frac{C}{N}
\end{equation}
for some constant $C:=C(T)$ independent of $N$. 
\begin{proof}
For $t\in [t_k, t_{k+1})$, $k=0, \cdots, N-1$ we have
\[ z^N(t)-u^N(t)=\int_{t_k}^t G(y^N(s)) dW(s).\]
Since $z^N(T)=u^N(T)$,  for any $p\geq 2$ the Burkholder-Davies-Gundy inequality implies 
\begin{align*} 				
\EE \Big( \sup_{t\in [0,T]} \|z^N(t)-u^N(t)\|_V^{2p} \Big) &
= \EE \Big( \sup_{0\leq k< N} \sup_{t\in [t_k, t_{k+1})} \|z^N(t)-u^N(t)\|_V^{2p}\Big) \nonumber \\
&\leq C_p \sum_{k=0}^{N-1} \EE\Big( \Big| \int_{t_k}^{t_{k+1}} \|G\big(s,y^N(s)\big)\|_{{\mathcal L}_2(K,V)}^2 ds \Big|^p \Big) \nonumber \\
&\leq C_p \sum_{k=0}^{N-1}\Big(\frac{T}{N}\Big)^p \sup_{t\in [0,T]} \big( K_0 + K_1 \EE(\|y^N(t)\|_V^{2p}) \big).
\end{align*}
Inequality \eqref{moments_splitting} concludes the proof of the first part of  \eqref{zN-uN_V}. Furthermore, 
\[ \sup_{t\in [0,T]} \EE\big(\|z^N(t)-u^N(t)\|_V^{2p}\big) = \sup_{0\leq k<N} \sup_{t\in [t_k, t_{k+1})} \EE\big(\|z^N(t)-u^N(t)\|_V^{2p}\big) .\]
A similar argument concludes the proof. 
\end{proof} 
\subsection{ A localized $L^2(\Omega)$ convergence}
Recall that  $X=\LL^4(D)\cap H$  is an interpolation space between $H$ and $V$ such that \eqref{interpol} holds. For every $M>0$, set
\begin{equation} 		\label{def_omegatilde}
\tilde{\Omega}_M(t):=\Big\{\omega \in \Omega \; : \;  \sup_{s\in [0,t]} \|u(s)(\omega)\|_X^4 \leq M\Big\} .
\end{equation}   
Note that once more, and unlike \cite{BeBrMi}, this set only depends on the solution $u$ of \eqref{2D-NS} and does not depend on the scheme.
Let $\tau_M=:\inf\{ t \geq 0 : \|u(t)\|_X^4 \geq M\} \wedge T$; then $\tau_M=T$ on $\Omega_M(T)$. The following  result 
improves Proposition 5.1 in
\cite{BeBrMi}. Note that in our case, both processes $u$ and $z^N$ have a.s. continuous trajectories.
\begin{prop} 		\label{prop5.1-mod}
Let $u_0$ be $V$-valued, ${\mathcal F}_0$-measurable such that $\EE(\|u_0\|_V^8)<\infty$ and suppose that $G$ satisfies the conditions {\bf (G1)} and {\bf (G2)}.  
 Then, there exist constants $C$ and $\widetilde{C(M)}$ such that
\begin{equation} 			\label{loc_cv_splitting}
\EE\Big(\! \sup_{t\in [0,\tau_M]} |z^N(t)-u(t)|_{\LL^2}^2 +\int_0^{\tau_M}\! \!\big[ \|u^N(t)-u(t)\|_V^2 + \|y^N(t)-u(t)\|_V^2\big] dt \Big) 
\leq \frac{C}{N}   e^{T \widetilde{C(M)}  },
\end{equation}
where 
\begin{equation}		\label{tildeC_splitting}
\widetilde{C(M)} :=    \frac{ 3^3 \bar{C}^2 }{2^5 \beta^3 \nu^3} M + C(\nu, L_1,\beta, \epsilon) , 
\end{equation}    
and   ${\bar C} $ is the constant defined in \eqref{interpol}  and $\beta <1 $. 
 \end{prop} 
 \begin{proof}
Let us apply the It\^o formula to $|z^N(t\wedge \tau_M)-u(t\wedge \tau_M)|_{\LL^2}^2$. 
This is possible even if $z^N$ and $u$ are not regular enough. Indeed, we can use the Yosida approximation $e^{-\delta A}\big(z^N(t)-u(t)\big)$
for some $\delta >0$, apply the It\^o
formula to this smooth processes, and then pass to the limit as $\delta \to 0$ 
(see e.g. \cite{ChuMil}, step 4 of the proof of Theorem 2.4).  
This implies
\begin{align}
|z^N(t\wedge \tau_M)-u(t\wedge \tau_M)|_{\LL^2}^2 = \sum_{i=1}^2 T_i(t\wedge \tau_M) + I(t\wedge \tau_M), 
\end{align} 
where 
\begin{align*}
T_1(t\wedge \tau_M)= &-2\int_0^{t\wedge \tau_M} \big\langle F(u^N(s)) - F(u(s)) \, , \, z^N(s)-u(s) \big\rangle ds, \\
T_2(t\wedge \tau_M)=&\int_0^{t\wedge \tau_M} \|G(y^N(s)) - G(u(s))\|_{{\mathcal L}_2(K,H)}^2 ds, \\
I(t\wedge \tau_M)=&2\int_0^{t\wedge \tau_M} \big(  z^N(s)-u(s)\, , \, \big[ G(y^N(s)) - G(u(s)) ] dW(s)\big).
\end{align*}
The Lipschitz property {\bf (G1)}  
 implies 
\begin{align}			\label{estim_T2}
T_2(t\wedge \tau_M) \leq & 2 L_1 \Big[ \int_0^{t\wedge \tau_M} |z^N(s)-u(s)|_{\LL^2}^2 ds 
+\int_0^{t\wedge \tau_M} |z^N(s)-y^N(s)|_{\LL^2}^2 ds \Big] . 
\end{align}
Furthermore, $T_1(t\wedge \tau_M) =  \sum_{i=1}^4 T_{1,i}(t\wedge \tau_M)$, where
\begin{align*}
T_{1,1}(t\wedge \tau_M)= &-2 \nu \int_0^{t\wedge \tau_M}  \langle Au^N(s) - Au(s)) \, , \, u^N(s)-u(s) \rangle ds, \\
T_{1,2}(t\wedge \tau_M)=& -2\nu  \int_0^{t\wedge \tau_M}  \langle Au^N(s) - Au(s)) \, , \, z^N(s)-u^N(s) \rangle ds, \\
T_{1,3}(t\wedge \tau_M)=&  -2 \int_0^{t\wedge \tau_M}  \langle B(u^N(s), u^N(s)) - B(u(s),u(s)) \, , \, u^N(s)-u(s) \rangle ds,\\
T_{1,4}(t\wedge \tau_M)=&  -2 \int_0^{t\wedge \tau_M}  \langle B(u^N(s), u^N(s)) - B(u(s),u(s)) \, , \, z^N(s)-u^N(s) \rangle ds. 
\end{align*}
The definition of $A$ implies that
\begin{equation}			\label{estim_T11}
T_{1,1}(t\wedge \tau_M)=  -2\nu \int_0^{t\wedge \tau_M}  |\nabla u^N(s)- \nabla u(s)|_{\LL^2}^2 ds . 
\end{equation}
 The Cauchy Schwarz and Young inequalities imply that for any $\beta_2 >0$,
 \begin{equation}		\label{estim_T12}
 T_{1,2}(t\wedge \tau_M) \leq\nu  \beta_2 \int_0^{t\wedge \tau_M} |\nabla u^N(s) - \nabla u(s)|_{\LL^2}^2  ds 
 + \frac{\nu}{\beta_2} \int_0^t \|z^N(s)-u^N(s)\|_V^2 ds. 
 \end{equation}
Using the inequality \eqref{B-B}, we deduce that for any $\beta_3 >0$ and $\epsilon >0$, 
\begin{align}		\label{estim_T13}
 &T_{1,3}(t\wedge \tau_M) \leq  2\nu \beta_3 \int_0^{t\wedge \tau_M} \big( |\nabla u^N(s) - \nabla u(s)|_{\LL^2}^2 
  + |u^N(s)-u(s)|_{\LL^2}^2\big) ds \nonumber \\
 &\qquad  +2  C_{\nu \beta_3} \int_0^{t\wedge \tau_M} \| u(s)\|_X^4\, |u^N(s)-u(s)|_{\LL^2}^2 ds\nonumber \\
  \;    \leq  &2\nu \beta_3 \! \int_0^{t\wedge \tau_M} \! \! |\nabla [u^N(s) -  u(s)]|_{\LL^2}^2 ds 
  + 2 (\nu \beta_3 +C_{\nu \beta_3} M) (1+\epsilon) \! \int_0^{t\wedge \tau_M}\! \!  |z^N(t)-u(t)|_{\LL^2}^2 ds \nonumber \\
   & 
 +  C(\nu, \beta_3, \epsilon)  \int_0^{t\wedge \tau_M}\! \big( 1+ \|u(s)\|_X^4)  |z^N(t)-u^N(t)|_{\LL^2}^2 ds.  
 \end{align}

Finally, since $B$ is bilinear, the H\"older inequality implies 
\begin{align*}		
T_{1,4}(t\wedge \tau_M)   &= 2 \int_0^{t\wedge \tau_N}  \Big[ \big\langle B(u^N(s)-u(s), z^N(s)-u^N(s))\, , \,  u^N(s)\big\rangle \\ 
&\qquad\qquad \qquad +
\langle B(u(s), z^N(s)-u^N(s))\, , \, u^N(s)-u(s)\rangle \big] ds
\\  
& \leq    2  \int_0^{t\wedge \tau_M}   \|u^N(s)-u(s)\|_X \big[ \|u^N(s)\|_X + \|u(s)\|_X \big]   \|z^N(s)-u^N(s)\|_V  ds .
\end{align*} 
 Using the interpolation inequality \eqref{interpol} and the  Young  inequality, we deduce 
that for any $\beta_4>0$, there exists a positive constant $C(\nu, \beta_4)$ such that 
\begin{align*} & \sqrt{\bar{C}} \big| \nabla \big(u^N(s)-u(s)\big) \big|_{\LL^2}^{\frac{1}{2}} 
|u^N(s)-u(s)|_{\LL^2}^{\frac{1}{2}}  \|z^N(s)-u^N(s)\|_{V} 
\big[ \|u^N(s)\|_X + \|u(s)\|_X\big] \\
&\quad \leq \nu \beta_4  \big| \nabla \big(u^N(s)-u(s)\big) \big|_{\LL^2}^2 +  C |z^N(s)-u(s)|_{\LL^2}^2 +
C |z^N(s)-u^N(s)|_{\LL^2}^2  \\
&\qquad + C(\nu, \beta_4)  
\big[ \|u^N(s)\|^2_X + \|u(s)\|^2_X \big]   \|z^N(s)-u^N(s)\|_V^2.
\end{align*}
This implies 
\begin{align} \label{estim_T14}
&T_{1,4}(t\wedge \tau_M)  \leq \; \nu \beta_4  \int_0^{t\wedge \tau_M} \!  \big|\nabla\big[ u^N(s)-u(s)\big] \big|_{\LL^2}^2 ds 
+  C \int_0^{t\wedge \tau_M} \!   |z^N(s)-u(s)|_{\LL^2}^2  ds  \nonumber \\
&\quad +  
C(\nu, \beta_4)  \int_0^{t\wedge \tau_M}  \!\! \Big(  |u^N(s)-z^N(s)|_{\LL^2}^2    + 
\big[ \|u^N(s)\|_X^2 + \|u(s)\|_X^2\big] \|z^N(s)-u^N(s)\|_V^2 \Big) ds .  
  \end{align}

Collecting the upper estimates \eqref{estim_T2}--\eqref{estim_T14}, we deduce that for   $\beta_2+2\beta_3+\beta_4 <2$  
and  $t\in [0,T]$, 
\begin{align}			\label{sup_Z^N-u}
& \sup_{0\leq s\leq t} |z^N(s\wedge \tau_N) - u(s\wedge \tau_N)|_{\LL^2}^2 +  \nu(2-\beta_2 -2\beta_3-\beta_4) 
 \int_0^{t\wedge \tau_N} |\nabla[u^N(s)-u(s)]|_{\LL^2}^2 ds   \nonumber \\
& \; \leq  R(t) + \sup_{s\in [0,t]} I(s\wedge \tau_N) + \Big[ 2  (1+\epsilon) C_{\nu \beta _3}M  + 2L_1 + C(\nu,\beta_3,\epsilon) \Big]
 \! \int_0^{t\wedge \tau_M} \!\!
|z^N(s)-u(s)|_{\LL^2}^2 ds, 
\end{align}
where,  gathering all error terms, we let
\begin{align*}
R(t)=&  \int_0^{t\wedge \tau_M} \Big(   2 L_1 |z^N(s)-y^N(s)|_{\LL^2}^2  +  C(\nu, \beta_2, \beta_3, \beta_4)  
  \|z^N(s)-u^N(s)\|_V^2 \Big) ds\\
& + C(\nu, \beta_3, \epsilon)   \Big\{ \int_0^{t\wedge \tau_M} \|u(s)\|_X^8 ds\Big\}^{\frac{1}{2}}  
 \Big\{ \int_0^{t\wedge \tau_M} |z^N(s)-u^N(s)|_{\LL^2}^4 ds \Big\}^{\frac{1}{2}} \\
&  +  C(\nu, \beta_4)   \int_0^{t\wedge \tau_M} \! 
\big[ \|u^N(s)\|_X^2 + \|u(s)\|_X^2\big] \|z^N(s)-u^N(s)\|_V^2 ds.
\end{align*}
The Cauchy-Schwarz inequality, the integrability property $\EE(\|u_0\|_V^8)<\infty$, 
and the upper estimates \eqref{bound_u}, \eqref{moments_splitting},
\eqref{zN-uN_V} and \eqref{zN-yN} imply the existence of some positive constant $C$ (depending on the parameter $\beta_i$ and
 $\epsilon$ such that  for every $N,M$,
\begin{equation} 				\label{error_R}
\EE \Big( \sup_{t\in [0,T]} R(t)\Big)  \leq \frac{C}{N}.
\end{equation}
Using the  Burkholder-Davis-Gundy inequality, then the Young inequality and the Lipschitz condition {\bf (G1)} 
on $G$, 
we deduce that for any 
$\delta >0$, 
\begin{align} 				\label{maj_I_splitting}
\EE\Big( &\sup_{t\in [s\leq t]} I(t)\Big) \leq  6 \EE\Big( \Big\{ \int_0^{t\wedge \tau_M} \|G(s,y^N(s)) - G(s,u(s))\|_{{\mathcal L}_2(K,H)}^2
 |u^N(s)-u(s)|_{\LL^2}^2 ds \Big\}^{\frac{1}{2}} \Big) \nonumber  \\
  \leq & 6 \EE\Big( \sup_{s\in [0,t\wedge \tau_M]} |z^N(s)-u(s)|_{\LL^2}  \Big\{ \int_0^{t\wedge \tau_M}
   L_1 |y^N(s)-u(s)|_{\LL^2}^2 ds \Big\}^{\frac{1}{2}}\Big).  \nonumber \\
  \leq & \delta \EE\Big( \sup_{s\in [0,t]} |z^N(s\wedge \tau_M)-u(s\wedge \tau_M)|_{\LL^2}^2 \Big)
  + C(\delta)  \EE\int_0^{t\wedge \tau_M} 
  |z^N(s)-u(s)|_{\LL^2}^2 ds \nonumber  \\
  & + C(\delta)   \EE\int_0^{t\wedge \tau_M} 
  |y^N(s)-z^N(s)|_{\LL^2}^2 ds  \nonumber \\
  \leq & \delta \EE\Big( \! \sup_{s\in [0,t]} |z^N(s\wedge \tau_M)-u(s\wedge \tau_M)|_{\LL^2}^2\! \Big)
  + C(\delta)  \EE\int_0^{t\wedge \tau_M} \!\!
  |z^N(s)-u(s)|_{\LL^2}^2 ds + \frac{C(\delta)}{N},
\end{align}
where the last inequality is a consequence of \eqref{zN-yN}. 
The upper estimates \eqref{sup_Z^N-u}--\eqref{maj_I_splitting} imply that if $\delta <1$ and 
 $\bar{C}(\beta) := 2- \beta_2 + 2\beta_3 + \beta_4 >0$,  
\begin{align} 			\label{Gronwall_splitting}
&(1-\delta) \EE\Big( \!\sup_{s\in [0,t]} |z^N(s\wedge \tau_M) - u(s\wedge \tau_M)|_{\LL^2}^2 \!\Big) +
\nu \bar{C}(\beta) \EE\int_0^{t\wedge \tau_M} \!|\nabla [u^N(s)-u(s)| |_{\LL^2}^2 ds  \nonumber \\
&\; \leq \frac{C}{ N} + \Big[ 2 (1+\epsilon)  C_{\nu \beta_3} M + C(L_1,\nu,\beta_3, \epsilon, \delta) \Big]
\EE \int_0^t |z^N(s\wedge \tau_M) - u(s\wedge \tau_M)|_{\LL^2}^2 ds. 
\end{align}

The Gronwall inequality (disregarding the second term on   the left hand side of \eqref{Gronwall_splitting}) proves that 
\[ \EE \Big( \sup_{s\in [0,t]} |z^N(s\wedge \tau_M)-u(s\wedge \tau_M)|_{\LL^2}^2\Big) \leq \frac{C}{N}  \exp\big(T\, \widetilde{C(M)}\big),\]
where $\widetilde{C(M)}$ is defined by \eqref{tildeC_splitting}.  Indeed, the term   $\beta_3< 1$ has to be chosen first since
the other positive constants $\beta_2$ and $\beta_4$, which have to satisfy $\beta_2+\beta_4 < 2(1-\beta_3)$,  
only appear in the "error" term $R(t)$. 
Using this inequality in \eqref{Gronwall_splitting},   the definition of $C_{\nu \beta_3}$ given in \eqref{def_Ceta}, 
\eqref{zN-uN_V} and \eqref{zN-yN},  and changing  the ratio  $\frac{1+\epsilon}{(1-\delta) \beta_3^3}$  into $\frac{1}{\beta^3}$ 
  for some $\beta \in (0,1)$,  
we conclude the proof of \eqref{loc_cv_splitting}. 
\end{proof}
\medskip

\subsection{Strong $L^2(\Omega)$ speed of convergence of the splitting scheme} 		\label{sec_strong_splitting}
Since $\tau_M=T$ on $\tilde{\Omega}_M$, we can rewrite \eqref{loc_cv_splitting} as 
\[ \EE\Big(1_{\tilde{\Omega}_M} \Big\{  \sup_{t\in [0,T]} |z^N(t)-u(t)|_{\LL^2}^2 
+\int_0^T \! \!\big[ \|u^N(t)-u(t)\|_V^2 + \|y^N(t)-u(t)\|_V^2\big] dt \Big\}\Big) 
\leq \frac{C}{N}  e^{T \tilde{C}(M)  }  .
\]
The previous section provides an upper bound  of the $L^2(\Omega)$ - norm of the maximal error on each time step 
localized on the set $\Omega_M$. In order to deduce the strong speed of convergence, we next have to analyze this error on the complement
of this localization set. 

The first step is described in the following simple upper estimates, which follow from the H\"older inequality. 
Hence, we  suppose that $\EE(\|u_0\|_V^{2q})<\infty$ and that $p$ and $q$ are conjugate exponents; then
\begin{equation}			\label{Holder_L2tilde}
\EE \Big( 1_{\tilde{\Omega}_M^c} \sup_{t\in [0,T]} |z^N(t)-u(t)|_{\LL^2}^2 \Big) \leq C \Big[\PP(\tilde{\Omega}_M^c)\Big]^{\frac{1}{p}}
\Big[ \EE \Big( \sup_{t\in [0,T]} \big( |z^N(t)|_{\LL^2}^{2q} + |u(t)|_{\LL^2}^{2q} \big) \Big) \Big]^{\frac{1}{q}}. 
\end{equation}
Note that we have the similar upper estimate
\begin{align*} \EE \Big(& 1_{\tilde{\Omega}_M^c}  \!\int_0^T \!\! \|z^N(t)-u(t)\|_V^2 dt \! \Big) \\
&   \leq C \Big[\PP(\tilde{\Omega}_M^c)\Big]^{\frac{1}{p}}
\Big[ \EE\! \!\int_0^T \!\!\big( \|z^N(t)-u^N(t)\|_V^{2q} + \| u^N(t)\|_V^{2q} + \|u(t)\|_V^{2q}\big) dt \Big]^{\frac{1}{q}}  
\leq C \Big[\PP(\tilde{\Omega}_M^c)\Big]^{\frac{1}{p}},
\end{align*}
where the last upper estimate  is deduced from \eqref{bound_u}, \eqref{moments_splitting} and \eqref{zN-uN_V}. 
Theorem \ref{strong_wp} shows that $\EE(\sup_{t\in [0,T]} \|u(t)\|_V^2)<\infty$. 
We next prove that $\EE(\sup_{t\in [0,T]} |z^N(t)|_{\LL^2}^{2q})$ is bounded by a constant independent of $N$.     
\begin{lemma}			\label{supt_ZN}
Let $p\geq 2$ and $u_0$ be $V$-valued, ${\mathcal F}_0$-measurable such that    $\EE(\|u_0\|_V^{2p+1})<\infty$. 
Suppose that $G$ satisfies the conditons {\bf (G1)} and {\bf (G2)}.    
Then there exists a positive constant $C_p$ such that for every integer $N\geq 1$,
\begin{align*}
\EE\Big( \sup_{t\in [0,T]} |z^N(t)|_{\LL^2}^{2p} +  \nu 
 \int_0^T \|z^N(t)\|_V^{2p} dt \Big)  \leq C_p   . 
\end{align*}
\end{lemma} 
 
 \noindent  The proof of this lemma is given in the Appendix, Section \ref{A1}.  \\

We next have to make sure that the threshold $M(N)$ is chosen to balance the upper estimates \eqref{loc_cv_splitting} and \eqref{Holder_L2tilde}.

\noindent {\bf Case 1: Linear growth diffusion coefficient.}
Suppose that $\EE(\|u_0\|_V^{2q})<\infty$. Then using 
 the Gagliardo-Nirenberg inequality \eqref{interpol}, we deduce 
\begin{align*}
\PP \big( \tilde{\Omega}_{M(N)}^c\big)  = &\PP\Big( \sup_{t\in [0,T]} \|u(t)\|_X^4 \geq M(N)\Big) 
 =   \PP\Big( \sup_{t\in [0,T]} \|u(t)\|_V^4 \geq \frac{4 M(N)}{\bar{C}^2} \Big)  \\
\leq & \frac{\bar{C}^q}{2^q M(N)^{\frac{q}{2}}} \EE\Big( \sup_{t\in [0,T]} \|u(t)\|_V^{2q}\Big) \leq \frac{C}{M(N)^{\frac{q}{2}}},
\end{align*}     
where the last upper estimate  is a consequence of \eqref{bound_u}. 
To balance the right hand sides of \eqref{loc_cv_splitting} and \eqref{Holder_L2tilde}, we choose $M(N)\to \infty$ as $N\to \infty$ such that
 \begin{equation} 			\label{constraint_lingrowth}
 \frac{1}{N} \exp( T \widetilde{C(M(N))})   \asymp C(q) \frac{1}{M(N)^{\frac{q-1}{2}}}.   
 \end{equation}   
Taking logarithms and using \eqref{tildeC_splitting}, this comes down to 
 \[ -\ln(N) +  2 (1+\epsilon) C_{\nu \beta}  M(N) T  \asymp -\frac{q-1}{2} \ln(M(N)), \]   
for some  $\beta\in (0,1)$   and $\epsilon \in (0,1)$, where $C_{\nu \beta}$ is defined in \eqref{def_Ceta}.  Let
\begin{equation} 			\label{M(N)_splitting_lingrowth}
M(N):= \frac{1}{2 (1+\epsilon) C_{\nu \beta} T} \Big[ \ln(N) -  
\frac{q-1}{2}  \ln\big(\ln(N)\big)\Big] \asymp C(\nu, \beta, \epsilon, T) \ln(N). 
\end{equation}
Then $M(N)\to \infty$ as $N\to \infty$ and \eqref{constraint_lingrowth} is satisfied. This yields the following result, where the first upper estimate
follows from \eqref{constraint_lingrowth} and \eqref{zN-uN_V}, while the second one is deduced from the fact that $y^N(t_k^+) = u^N(t_{k+1}^-)$. 
\begin{theorem}		\label{strong_splitting_LG}
Suppose that $u_0$ is $V$-valued such that  $\EE(\|u_0\|_V^{2q+1})<\infty$  for some $q\geq 4$  and that $G$ satisfies conditions 
 {\bf (G1)} and {\bf (G2)}.     
Then there exists a constant $C>0$ such that 
\begin{align} 			\label{strong_cv_splitting_LG1}
\EE\Big( &\sup_{t\in [0,T]}  \big[ |z^N(t)-u(t)|_{\LL^2}^2 + |u^N(t)-u(t)|_{\LL^2}^2 \big]   \nonumber \\
&+ \int_0^T \big[ \|z^N(s)-u(s)\|_V^2 + \|u^N(s)-u(s)\|_{V}^2 + \|y^N(s)-u(s)\|_V^2\big]  ds \! \Big) \leq   \frac{C}{\ln(N)^{\frac{q-1}{2}}},   \\
\EE\Big( &\sup_{k=1, \cdots, N} \big[ |u^N(t_k^+)-u(t_k)|_{\LL^2}^2 + |y^N(t_k^+)-u(t_k)|_{\LL^2}^2 \big] \Big) \leq 
\frac{C}{\ln(N)^{\frac{q-1}{2}}}.   
				\label{strong_dis_splitting_LG}
\end{align} 
\end{theorem}
\begin{remark}
Note that  if $u_0$ is a deterministic element of $V$, or more generally if $\|u_0\|_V$ has moments of all orders 
(for example if $u_0$ is a 
$V$-valued Gaussian random variable independent of the noise $W$),  then the speed of convergence of the current splitting
scheme is any negative power of $\ln(N)$. 
\end{remark}
\medskip 

\noindent{\bf Case 2: Additive noise.} 
Suppose that $G(u):=G\in {\mathcal L}_2(K,V)$, that is the noise is additive, or more generally that $G$ satisfies the conditions 
{\bf (G1)} and {\bf (G2)}   
 with $K_1=0$, that is $\|G(u)\|^2_{{\mathcal L}_2(K,V)}\leq K_0$. 
Then for any constant $\alpha >0$
using an exponential Markov inequality and the Gagliardo-Niremberg inequality \eqref{interpol}, we deduce
\begin{equation}				\label{upper_exp_splitting}
 \PP(\tilde{\Omega}_M^c) \leq  \PP\Big( \sup_{t\in [0,T]}  \| u(t)\|_X^2 \geq  \sqrt{M}  \Big)  \leq
\PP\Big( \sup_{t\in [0,T]} \|u(t)\|_V^2 \geq \frac{2 \sqrt{M}}{\bar{C}}\Big) .
\end{equation}    
 We next prove that for an additive noise, or in a slightly more general setting, for $\alpha >0$ small enough, 
$\EE[ \exp(\alpha \sup_{t\in [0,T]} \| u(t)\|_{V}^2)]<\infty$. In the case of an additive noise, this result is a particular case of 
\cite{HaiMat}, Lemma A.1. In this reference, the periodic Navier Stokes equation is written in vorticity formulation $\xi(t) = \partial_1 u_2(t) -
\partial_2 u_1(t)$. The velocity $u$ projected on divergence free fields can be deduced from $\xi$ by the Biot-Savart kernel, and $|\nabla u(t)|_{\LL^2}
\leq C |\xi(t)|_{\LL^2}$. We extend this result to the case of a more general diffusion coefficient whose Hilbert-Schmidt norm is bounded. 
 Recall that the Poincar\'e inequality implies the existence of a  constant $\tilde{C}>0$ such that 
if we set $|||u|||^2:= |\nabla u|^2_{\LL^2} + |Au|^2_{\LL^2}$ for $u\in \mbox{\rm Dom}(A)$, then we have
\begin{equation}			\label{Poincare}
\|u\|^2_V= |  u |^2_{\LL^2} + | \nabla u |^2_{\LL^2} \leq  \tilde{C} |||u|||^2  .
\end{equation} 

\begin{lemma}   			\label{exp_mom}
Let $u_0\in V$ and $G$ satisfy conditions  {\bf (G1)} and {\bf (G2)}  
with $K_1=0$, that is $\|G(t,u)\|_{{\mathcal L}_2(K,V)}^2 \leq K_0$. 
Then the solution $u$ to \eqref{2D-NS} satisfies 
\begin{equation}					\label{moments_exp}
 \EE \Big\{ \exp \Big[\alpha \Big( \sup_{0\leq s\leq T} \|u(s)\|_V^2 \Big)   
 + \nu \int_0^T \! |Au(s)|_{\LL^2}^2 ds \Big) \Big] \Big\} \leq 
 3 \exp\big( \alpha[ \|u_0\|_V^2 + T K_0]\big)
\end{equation}
for $\alpha \in (0, \alpha_0]$ and $\alpha_0= \frac{\nu}{4 K_0 \tilde{C}}$, where $\tilde{C}$ is defined in \eqref{Poincare}. 
\end{lemma}
In order to make this paper as self-contained as possible, we prove this lemma in the appendix, section \ref{A2}.

Let $u_0\in V$; then for $\alpha_0=\frac{\nu}{4K_0 \tilde{C}}$, 
the inequalities \eqref{upper_exp_splitting} and  \eqref{moments_exp} imply
\[ \PP(\tilde{\Omega}_M^c) 
\leq 3 e^{\alpha_0 K_0 T}\,  \exp\Big( -  2 \alpha_0 \frac{\sqrt{M}}{\bar{C}}\Big) ,\]
where   $\bar{C}$ is defined by 
\eqref{interpol}.   We then have to choose $M(N)\to \infty$ as
$N\to \infty$ to balance the right hand sides of \eqref{loc_cv_splitting} and \eqref{Holder_L2tilde}, that is such that for some $p>1$, 
\[    \exp\Big( -\frac{2 \alpha_0 \sqrt{M(N)}}{p \bar{C} }\, T \Big)  \asymp  c_2 \frac{T}{N}  \exp\big( \widetilde{C(M(N))} T\big)\]
for some positive constant  $c_2$. Taking logarithms, we look for $M(N)$ such that 
\[ - \frac{2\alpha_0}{p\bar{C}} \sqrt{M(N)} \asymp -\ln(N) +   2(1+\epsilon) C_{\nu \beta} M(N) T ,  \]
where  $\beta \in (0, 1)$,  $\epsilon >0$ and $C_{\nu \beta}$ is defined by \eqref{def_Ceta}.   
Set $X=\sqrt{M(N)}$,
  $a_2= 2(1+\epsilon) C_{\nu \beta}  T $ 
and  $a_1= \frac{2 \alpha_0}{p\bar{C}} $.  
We have to solve the equation $a_2 X^2 + a_1 X -\ln(N)=0$. The positive root of this polynomial is
 equal  to $\sqrt{\frac{\ln (N)}{a_2}} + O(1)$   as $N\to \infty$ and 
 $\exp\big(-\frac{2 \alpha_0 \sqrt{M(N)}}{p \bar{C} }\big) \asymp C \exp\big(- 2 \alpha_0 \frac{\sqrt{\ln (N)}}{\sqrt{a_2} p \bar{C} }\big)$.  
Thus, we deduce the following rate of convergence of the splitting scheme.
\begin{theorem}		\label{strong_cv_splitting_add}
Let $u_0\in V$ and $G$ satisfy the conditions {\bf (G1)} and {\bf (G2)}  
with $K_1=0$, that is $\|G(s,u)\|_{{\mathcal L}(K,V)}^2 \leq K_0$.
Then 
\begin{align} 			\label{rate_splitting_add}
&\EE\Big( \sup_{t\in [0,T]}  \big[ |z^N(t)-u(t)|_{\LL^2}^2 + |u^N(t)-u(t)|_{\LL^2}^2 \big]    \\
&\; + \int_0^T \big[ \|z^N(s)-u(s)\|_V^2 + \|u^N(s)-u(s)\|_{V}^2 + \|y^N(s)-u(s)\|_V^2\big]  ds \! \Big) \leq C  e^{-\gamma \sqrt{\ln(N)}},  \nonumber\\
\EE\Big( &\sup_{k=1, \cdots, N} \big[ |u^N(t_k^+)-u(t_k)|_{\LL^2}^2 + |y^N(t_k^+)-u(t_k)|_{\LL^2}^2 \big] \Big) \leq C  e^{-\gamma \sqrt{\ln(N)}}, 
\label{rate_split_discrete_add}
\end{align} 
where 
 \[ \gamma <  \frac{\alpha_0 }{\bar{C}^2} \sqrt{\frac{2^9 \nu^3}{3^3 T}}. \] 
\end{theorem}

 Note that when  $\nu$ increases,  the upper bound of the exponent $\gamma$ increases. 
\begin{remark}
Note that the statements of Theorems \ref{strong_splitting_LG} and \ref{strong_cv_splitting_add} are valid
 if the diffusion coefficient $G$ depends on the time parameter $t\in [0,T]$ and satisfies the global growth and Lipschitz versions
of  {\bf (G1)} and {\bf (G2)}.   
We have removed the time dependence of $G$ 
to focus on the main arguments used to obtain strong convergence results.
\end{remark}

\section{Euler time schemes} 	 \label{sec_Euler}
\subsection{Description of the fully implicit scheme and first results}   \label{known_Euler}

In this section, we have to be more specific in the definition of the noise. Let ${\mathcal K}$ be a Hilbert space, $Q$ be a trace-class
operator in ${\mathcal K}$ and $W:=(W(t), t\in [0,T])$ be a ${\mathcal K}$-valued Wiener process with covariance $Q$. 
Let $K=Q^{\frac{1}{2}} {\mathcal K}$ denote the RKHS of the Gaussian process $W$. Let ${\mathcal G}:{\mathcal K}\to H$ be a linear operator and 
suppose that analogs of conditions {\bf (G1)} and
{\bf (G2)} are satisfied with ${\mathcal G}$ instead of $G$, and the operator norms ${\mathcal L}({\mathcal K},H)$ 
(resp. ${\mathcal L}({\mathcal K},V)$)
instead of  the Hilbert-Schmidt norms ${\mathcal L}_2(K,H)$  (resp. ${\mathcal L}_2({\mathcal K},V)$), with constants 
$\bar{K}_i$, $i=0,1$ and $\bar{L}_1$. 
Then the diffusion coefficient  $G={\mathcal G}\circ Q^{-\frac{1}{2}}$ satisfies conditions {\bf (G1)} and {\bf (G2)} 
with constants $K_i=\mbox{\rm Trace} (Q) \bar{K}_i$ and $L_1=\mbox{\rm Trace}(Q) \bar{L}_1$.

Let us first recall the fully implicit time discretization scheme of the stochastic 2D Navier-Stokes introduced by E.~Carelli and A.~Prohl in \cite{CarPro}. 
 As in the previous section, let  $t_k=\frac{kT}{N}$, $k=0, \cdots, N,$ denote the time grid. 
When studying a space time discretization using finite elements,
 one needs to have a stable pairing of the velocity and the pressure which
 satisfy the discrete LBB-condition (see e.g. \cite{CarPro}, page 2469 and pages 2487-2489). Stability issues are crucial and the pressure has
 to be discretized together with the velocity. In this section, our aim is to obtain bounds for the strong error of an Euler time scheme. Thus, as
 in the previous section, we may define the scheme for the velocity projected on divergence free fields (see \cite{CarPro}, Section 3). 
 \smallskip
 
\noindent {\bf Fully implicit Euler scheme} 
{\it Let $u_0$ be a $V$-valued,  ${\mathcal F}_0$-measurable  
random variable and set $u_N(t_0)=u_0$. 
For $k=1, \cdots, N,$ find
 $u_N(t_k) \in  V $ such that $\PP$  a.s. for  all  
 $\phi \in V$, 
\begin{align}    \label{full-imp1}
\big( u_N(t_k) - u_N(t_{k-1}) , \phi \big) +&  \frac{T}{N} \Big[ \nu \big( \nabla u_N(t_k) , \nabla \phi\big) + \big\langle B(u_N(t_k), u_N(t_k)), 
\phi\big\rangle \Big]\nonumber  \\
&\qquad \qquad \qquad 
= \big( G( u_N(t_{k-1})) \,  \Delta_k W , \phi\big) , 
\end{align}
 where $\Delta_k W = W(t_k) - W(t_{k-1})$. } \\

In the study of the Euler discretization schemes, we will need some H\"older regularity of the solution.
   This is proved by means of semigroup theory;  see 
 \cite{Pri},  Proposition 3.4 and \cite{CarPro}, Lemma 2.3. 
  \begin{prop}
  Let $u_0$ be  ${\mathcal F}_0$-measurable  such that $\EE(\|u_0\|_V^{2p}) <\infty$ for some $p\in [2,4]$. 
  Let $G$ satisfy conditions {\bf (G1)} and {\bf (G2)}.   
  Then for $\eta \in (0,\frac{1}{2})$, we have
  \begin{align}
  \EE\big( \|u(t)-u(s)\|_{\LL^4}^p\big) \leq C\; |t-s|^{\eta p}, \label{HolderL4}\\
  \EE\big( \|u(t)-u(s)\|_V^p \big) \leq C \; |t-s|^{\frac{\eta p}{2}}.   \label{HolderV}
  \end{align}
  \end{prop}

\par Let us recall Lemma 3.1 in \cite{CarPro}, which proves moment estimates of the solution to \eqref{full-imp1}. 
Note that here only dyadic moments are computed because of the induction argument which relates two consecutive dyadic
numbers (see step 4 of the proof of Lemma 3.1  in \cite{BrCaPr}).  
\begin{lemma}  \label{reg_scheme}
Let  $u_0$ be ${\mathcal F}_0$-measurable such that $\EE\big( \|u_0\|_V^{2^q} \big)<\infty$    for some integer $q\in [2,\infty)$.
 Assume that $G$ satisfies the conditions {\bf (G1)}  and {\bf (G2)}.  
Then there exists a $\PP$ a.s. unique sequence of solutions  $\big\{u_N(t_k)\}_{k=1}^N$  of \eqref{full-imp1}, 
such that each random variable $u_N(t_k)$ is 
 ${\mathcal F}_{t_k}$-measurable and satisfies:
\begin{equation}   \label{moments_scheme}
 \sup_{N\geq 1}  \EE\Big( \max_{1\leq k \leq N} \|u_N(t_k)\|_V^{2^q} + 
 \nu \frac{T}{N} \sum_{k=1}^N \| u_N(t_k)\|_V^{2^q-2} |Au_N(t_k)|_{\LL^2}^2 
 \Big) \leq C(T,q), 
  \end{equation} 
 where $C(T,q)$ is a constant which depends on $T$, the constants $K_i$,  $i=0,1$   in conditions {\bf (G1)} 
 and {\bf (G2)}, and also depends on $\EE(\|u_0\|_V^{2^q})$.
\end{lemma}
For $k=0, \cdots, N$, let $e_k:=u(t_k) - u_N(t_{k})$ denote the error of this scheme (note that $e_0=0$). 
Then for any $\phi \in V$ and $j=1, \cdots, N$,  we have
\begin{align}											\label{e-e}
(e_{j} - e_{j-1}, \phi) & +  \int_{t_{j-1}}^{t_j}  \Big[ \nu \big( \nabla u(s) - \nabla u_N(t_{j})\, ,\,  \nabla \phi) 
+ \big\langle  B(u(s)) - B(u_N(t_{j})) \, ,\,  \phi\big\rangle \Big] ds \nonumber  \\
&=
\Big(\phi \, , \,  \int_{t_{j-1}}^{t_j} \big[ G( u(s)) - G( u_N(t_{j-1}))\big] dW(s)  \Big).
\end{align}

\subsection{A localized convergence result} 
The first result states a localized upper bounds  of the error terms. 
This is due to the nonlinear term, but unlike \cite{CarPro}, it depends on $u$
and not on $u_N$.  Given $M>0$ and $k=1, \cdots, N$, set
\begin{equation}				\label{Omegak}
\Omega_k^M:=\Big\{ \omega \in \Omega \, :\, \max_{1\leq j\leq k} |\nabla u(t_j)|_{\LL^2}^2 \leq M\Big\} \in {\mathcal F}_{t_k}.
\end{equation} 
The following proposition  is one of the main results of this section. The modification with respect to Theorem 3.1 in \cite{CarPro} is the
localization set which does not depend on the approximation. This will be crucial to obtain a speed of  $L^2(\Omega)$- strong convergence, 
and not only that the scheme converges in probability.
\begin{prop} \label{Loc_cv_Euler} 
Let $G$ satisfy  the growth and Lipschitz conditions {\bf (G1)} and {\bf (G2)}.  
Let $u_0$ be such that $\EE(\|u_0\|_V^8) <\infty$. 
Then for $\Omega_k^M$  defined by \eqref{Omegak} and $N$ large enough, we have for every $k=1, \cdots, N$:
\begin{equation}				\label{loc_moments}
\EE\Big( 1_{\Omega_{k-1}^M } \max_{1\leq j \leq k}  \Big[  |e_j|_{\LL^2}^2   + {\nu} \frac{T}{N}
\sum_{j=1}^k   |\nabla e_j|_{\LL^2}^2  \Big] \Big) \leq C \; \exp\big[C_1(M) T\big] \;  \Big(\frac{T}{N}\Big)^\eta , 
\end{equation}
for  some constant $C>0$, $\eta \in (0, \frac{1}{2} )$,   and   
\begin{equation}				\label{def_C1_Euler}
C_1 (M) = \frac{ (1+\bar{\epsilon}) \bar{C}^2 } {2 \nu} \, M + C(\bar{\epsilon})\, L_1 , 
\end{equation}   
 where  $\bar{C}$ is defined in \eqref{interpol},  and $\bar{\epsilon}$ is arbitrary close to 0.
\end{prop}
\begin{proof}
We  follow the scheme of the arguments  in \cite{CarPro}, pages 2480-2484, but the upper estimate of the duality involving the
difference of the bilinear terms is dealt with differently, which leads to a different localization set. Furthermore, in order to describe
the strong speed of convergence of the scheme, we need a more precise control
of various constants appearing in some upper estimates. Hence we give a detailed proof below.\\
{\bf Step 1:  Upper estimates for the bilinear term}\\
 Let us consider the duality between the difference of bilinear terms and $e_{j}$, that is the  upper estimate of  $\int_{t_{j-1}}^{t_j} \langle
 B(u(s)) - B(u_N(t_{j})) \, , \, e_{j}\big\rangle ds$.  For every $s\in (t_{j-1}, t_j]$, using the bilinearity of $B$ 
and the antisymmetry property  \eqref{B}, 
we deduce 
\begin{equation} \label{dif_B-B}
  \big\langle B\big( u(s), u(s) \big) -  B\big( u_N(t_j), u_N(t_j)\big)  \, , \, e_j \big\rangle = \sum_{i=1}^3 T_i(s),
  \end{equation}
where, since  $\big\langle  B\big( v,  u_N(t_j)  \big)  \, , \, e_j \big\rangle = \big\langle  B\big( v,  u(t_j)  \big)  \, , \, e_j \big\rangle$ for every $v\in V$,
\begin{align*}
T_1(s):= & \big\langle  B\big( e_j ,  u_N(t_j)  \big)  \, , \, e_j \big\rangle = \big\langle B\big(e_j , u(t_j)\big)  \, , \, e_j \big\rangle,\\ 
T_2(s):= &\big\langle  B\big(  u(s) -u(t_j) , u(t_j) \big) \, , \, e_j \big\rangle, \\
T_3(s):= & \big\langle B\big(u(s) , u(s)-u(t_j)\big) \, , \, e_j \big\rangle = - \big\langle B\big(u(s) , e_j\big) \, , \, u(s)-u(t_j) \big\rangle.
\end{align*}

Note that, unlike the first formulation of $T_1(s)$, the second  one only depends on the error and on the solution to \eqref{2D-NS}, 
and not on the approximation scheme. The H\"older  inequality and \eqref{interpol} yield  
 for every $\delta_1 >0$ and $\bar{C}$ defined in the interpolation
inequality \eqref{interpol} 
\begin{align*} 				 
 \int_{t_{j-1}}^{t_j} | T_1(s) |\, ds &\leq \bar{C}  \frac{T}{N} \,  |  e_j|_{\LL^2}  \, |\nabla e_j|_{\LL^2} \,   | \nabla u(t_j)|_{\LL^2}  \\ 
& \leq \delta_1 \nu \frac{T}{N} \, |\nabla e_j|_{\LL^2}^2  + \frac{\bar{C}^2}{4\delta_1 \nu} \, \frac{T}{N} \,  |e_j|_{\LL^2}^2 \, |\nabla u(t_j)|_{\LL^2}^2, 
\end{align*}
where the last upper estimate follows from the Young inequality (with conjugate exponents 2 and 2).
A similar argument using the H\"older and Young inequalities with exponents  4, 4 and 2 imply that for any $\delta_2 >0$ and $\gamma_2>0$, 
\begin{align*}
|T_2(s)| \leq \delta_2 \nu |\nabla e_j|_{\LL^2}^2 + \gamma_2 |e_j|^2_{\LL^2} + 
 C(\nu, \delta_2, \gamma_2)   \|u(t_j)-u(s)\|_{\LL^4}^2 |\nabla u(t_j)|_{\LL^2}^2 . 
\end{align*}
Using the Cauchy-Schwarz inequality   we deduce 
\begin{align*}			
 \int_{t_{j-1}}^{t_j} \!\! | T_2(s) |\, ds \leq  \;  &\delta_2 \nu  \frac{T}{N}  |\nabla e_j|_{\LL^2}^2  
 + \gamma_2 \frac{T}{N}    |e_j|_{\LL^2}^2  \\ 
& \; +  
 C(\nu, \delta_2, \gamma_2)   |\nabla u(t_j)|_{\LL^2}^2 \int_{t_{j-1}}^{t_j}\!\!  \|u(t_j)-u(s)\|_{\LL^4}^2 ds . 
\end{align*} 
Similar computations using  the H\"older and Young  inequalities imply 
\begin{align}				\label{maj_T3}
\int_{t_{j-1}}^{t_j} | T_3(s) |\, ds \leq & \; \delta_3 \nu \frac{T}{N} |\nabla e_j|_{\LL^2}^2  
+ \frac{1}{4\nu \delta_3}  \int_{t_{j-1}}^{t_j} \|u(s)\|_{\LL^4}^2 \, \|u(t_j)-u(s)\|_{\LL^4}^2 ds
\end{align}
for any $\delta_3>0$. Note that 
\[ 
 \nu \int_{t_{j-1}}^{t_j} \!\! \big( \nabla(u(s) - u_N(t_j))\, , \, \nabla e_j \big)  ds =  \nu  \frac{T}{N}  |\nabla e_j|_{\LL^2}^2 + \nu  \int_{t_{j-1}}^{t_j} \!\!
\big( \nabla(u(s)-u(t_j)) ,  \nabla e_j\big) ds. 
\]
Using the Cauchy-Schwarz and Young inequalities, we deduce 
\[ \nu \int_{t_{j-1}}^{t_j} \!\!
\big| \big( \nabla(u(s)-u(t_j)) ,  \nabla e_j\big)\big|  ds  \leq 
\delta_0 \nu \frac{T}{N} |\nabla e_j|_{\LL^2}^2 + \frac{\nu }{4\delta_0 } \int_{t_{j-1}}^{t_j} \big|\nabla \big(u(s) - u(t_j) \big)\big|_{\LL^2}^2 ds 
\] 
for any $\delta_0>0$.  Hence using the above upper estimates in \eqref{e-e} with $\phi = e_j$, we deduce
\begin{align}						\label{(e-e,e)}
\big( e_j-e_{j-1}\, , e_j\big) & + \nu \frac{T}{N} |\nabla e_j|_{\LL^2}^2 \leq  \nu \sum_{r=0}^3 \delta_r \frac{T}{N} |\nabla e_j|_{\LL^2}^2 
+ \Big( \gamma_2 + \frac{\bar{C}^2}{4\delta_1 \nu} |\nabla u(t_j)|_{\LL^2}^2 \Big) \, \frac{T}{N} \, |e_j|_{\LL^2}^2  \nonumber \\
& + \sum_{l=1}^3 \tilde{T_j}(l) + \Big( \int_{t_{j-1}}^{t_j} \big[ G(u(s)) - G( u_N(t_{j-1}))\big] dW(s) \, , \, e_j\Big),
\end{align}
where 
\begin{align*}
\tilde{T_j}(1) =  & \frac{\nu }{4 \delta_0 } \int_{t_{j-1}}^{t_j} \big|\nabla \big(u(s) - u(t_j) \big)\big|_{\LL^2}^2 ds, \\
\tilde{T_j}(2) =  &  
 C(\nu, \delta_2, \gamma_2)   \, |\nabla u(t_j)|_{\LL^2}^2 \int_{t_{j-1}}^{t_j} \|u(s)-u(t_j)\|_{\LL^4}^2 ds, \\
\tilde{T_j}(3) =  &  \frac{1}{4\nu \delta_3} \int_{t_{j-1}}^{t_j} \|u(s)\|_{\LL^4}^2  \|u(t_j)-u(s)\|_{\LL^4}^2 ds. 
\end{align*}
Using the time regularity \eqref{HolderV}  with $p=2$, we deduce 
\begin{align} 		\label{bound_ET1}
\EE \big(  \tilde{T_j}(1)\big) & \leq  C  \frac{\nu }{4 \delta_0 } \Big( \frac{T}{N} \Big)^{1+\eta}.
\end{align}
The Cauchy-Schwarz inequality, \eqref{bound_u} with $p=2$ and \eqref{HolderL4} imply
\begin{align} 				\label{bound_ET2}
\EE \big(  \tilde{T_j}(2) \big) &\leq\,  C(\nu, \delta_2, \gamma_2)  \Big(\frac{T}{N}\Big)^{1+2\eta}, \\
\EE \big( \tilde{T_j} (3) \big) &\leq C \, \frac{1}{4\nu \delta_3}  \, \Big(\frac{T}{N}\Big)^{1+2\eta}.  \label{bound_ET3}
\end{align}
\bigskip

\noindent {\bf Step 2: Localization}\\
In order to use a discrete version of the Gronwall lemma to upper estimate $|e_j|_{\LL^2}^2$, due to the factor 
$|\nabla u(t_j)|_{\LL^2}^2$ on   the RHS of \eqref{(e-e,e)}, we have to localize on  the random set $\Omega_{j-1}^M$
defined in \eqref{Omegak}. The shift of index is due to the fact that, in order to deal with the stochastic integral,  we have
to make sure that the localization set is ${\mathcal F}_{t_{j-1}}$-measurable. 
This set  depends on $j$, but we will need to add the localized inequalities \eqref{(e-e,e)} and take expected values.

Note that for $1\leq j\leq k$, $\Omega_k^M \subset \Omega_j^M$.  Hence, since $e_0=0$,  as proved in  \cite{CarPro}, estimate  (3.25), we have 
\begin{align} \label{maj_e-e}
 \max_{1\leq j\leq k} \sum_{l=1}^j 1_{\Omega^M_{l-1}} \Big( |e_l |_{\LL^2}^2
- |e_{l-1}|_{\LL^2}^2 \Big) & = 
\max_{1\leq j\leq k}  \Big( 1_{\Omega^M_{j-1}} |e_j|_{\LL^2}^2 
+ \sum_{l=2}^j  \big( 1_{\Omega_{l-2}^M} - 1_{\Omega_{l-1}^M}\big) |e_{l-1}|_{\LL^2}^2\Big) \nonumber \\
 &  \geq \max_{1\leq j\leq k}1_{\Omega^M_{j-1} } |e_j|_{\LL^2}^2 .
\end{align}
Thus, we will localize $|e_j|_{\LL^2}^2$ on the set $\Omega^M_{j-1} $ and - shifting the index by one -
control some "error term" $|e_j-e_{j-1}|_{\LL^2}^2$ localized on the same set. Note that this
localization set only depends on the projection of the solution $u$ of equation \eqref{2D-NS} on divergence free fields, and not on its approximation. 

Adding the inequalities \eqref{(e-e,e)} with $\phi = e_j$ localized on the set $\Omega_{j-1}^M$,   using  $e_0=0$ and the identity 
$(a,a-b) = \frac{1}{2} \big[ |a|_{\LL^2}^2 - |b|_{\LL^2}^2 + |a-b|_{\LL^2}^2\big]$, we deduce for $k=1, \cdots, N$
\begin{align*}
\max_{1\leq j\leq k} & \Big(  \frac{1}{2}  1_{\Omega^M_{j-1}} |e_j|_{\LL^2}^2 +  
\frac{1}{2} \sum_{l=1}^j 1_{\Omega_{l-1}^M} |e_l-e_{l-1}|_{\LL^2}^2 \Big) \\
& \leq \frac{1}{2} \Big(  \max_{1\leq j\leq k} \sum_{l=1}^j 1_{\Omega^M_{l-1}} \Big( |e_l|_{\LL^2}^2
- |e_{l-1}|_{\LL^2}^2 \Big)  +   \sum_{l=1}^j 1_{\Omega_{l-1}^M} |e_l-e_{l-1}|_{\LL^2}^2  \Big)\\
&\leq  \max_{1\leq j\leq k}  \sum_{1\leq l\leq j}  1_{\Omega^M_{l-1}} \big(e_l-e_{l-1}, e_l \big)   . 
\end{align*}
The upper estimates \eqref{(e-e,e)}  for $j=1, \cdots, k$ imply  for any $\epsilon >0$ 
\begin{align}					\label{maj_max} 
\max_{1\leq j\leq k} &\Big[  \frac{1}{2}  1_{\Omega^M_{j-1}} |e_j|_{\LL^2}^2 +  \sum_{l=1}^j 1_{\Omega_{l-1}^M} |e_l-e_{l-1}|_{\LL^2}^2 
  +   \nu \big(1-\sum_{r=0}^3 \delta_r \big)\, \frac{T}{N} \sum_{l=1}^j 1_{\Omega_{l-1}^M} |\nabla e_l |_{\LL^2}^2  \Big] \nonumber \\
& \leq   \Big[ \gamma_2 + (1+\epsilon)    \frac{\bar{C}^2 \, M}{4\delta_1 \nu}  \Big]  \, \frac{T}{N} \, 
\sum_{j=1}^k 1_{\Omega_{j-1}^M}  |e_{j}|_{\LL^2}^2   
+  \sum_{i=1}^3 \sum_{j=1}^k \tilde{T_j}(i)  \nonumber \\
&+ C(\nu, \delta_1, \epsilon)  \, \frac{T}{N} \,\sum_{j=1}^k 1_{\Omega_{j-1}^M}  |e_j|_{\LL^2}^2 
 |\nabla \big[ u(t_j)- u(t_{j-1}) \big] |_{\LL^2}^2  +M_k(1) + M_k(2),
\end{align}
where
\begin{align*}
M_k(1)= & \sum_{j=1}^k 1_{\Omega_{j-1}^M} \Big(  e_{j-1}\, , \, \int_{t_{j-1}}^{t_j} \big[ G(u(s)) - G(u_N(t_{j-1}))\big] dW(s) \Big), \\
M_k(2)=& \sum_{j=1}^k 1_{\Omega_{j-1}^M} \Big( e_j-e_{j-1} \, , \, \int_{t_{j-1}}^{t_j} \big[ G(u(s)) - G(u_N(t_{j-1}))\big] dW(s)  \Big).
\end{align*}
The inequalities \eqref{bound_ET1}-- \eqref{bound_ET3} imply the existence of a constant $C$ depending on $T$, $\nu$, $\delta_i$, $i=0, \cdots 3$
and $\gamma_2$ such that
\begin{equation}			\label{error_term_bilin}
\sum_{i=1}^3  \sum_{j=1}^N \EE \big( \tilde{T_j}(i) \big)  \leq C \; \Big( \frac{T}{N}\Big)^\eta. 
\end{equation}
The Cauchy-Schwarz inequality, \eqref{bound_u} and  \eqref{moments_scheme} for $p=q=2$, and the time regularity \eqref{HolderV} for $p=4$
imply the existence of a constant $C$ such that 
\begin{equation}			\label{error_localization}
 \frac{T}{N} \,\sum_{j=1}^N \EE  \Big(   |e_j|_{\LL^2}^2 
|\nabla \big[ u(t_j))- u(t_{j-1}) \big] |_{\LL^2}^2 \Big)  \leq C \Big( \frac{T}{N}\Big)^\eta      . 
\end{equation}
We next upper estimate $\EE\big( \max_{1\leq k\leq N} M_k(2)\big)$. The Cauchy-Schwarz inequality, the It\^o isometry and then
the Young inequality   
imply that for any $\tilde{\delta}_2 >0$ 
\begin{align*}			
&  \EE\Big( \max_{1\leq j\leq k} M_j(2) \Big) \leq  
\sum_{j=1}^k \Big\{ \EE  \big(1_{\Omega_{j-1}^M}  \, |e_j - e_{j-1}|_{\LL^2}^2\big) \Big\}^{\frac{1}{2}}  \\ 
&\qquad \qquad  \times  \Big\{ \EE \Big(   1_{\Omega_{j-1}^M }\int_{t_{j-1}}^{t_j} \| G(u(s)) - G( u_N(t_{j-1}) \|_{{\mathcal L}_2(K,H)}^2 ds\Big) 
 \Big\}^{\frac{1}{2}}		\\ 
&\;  \leq   \tilde{\delta}_2  \sum_{j=1}^k    \EE\big(1_{\Omega_{j-1}^M} \, |e_j- e_{j-1}|_{\LL^2}^2 \big) \\
&\qquad \qquad + \frac{1}{4\tilde{\delta}_2} \sum_{j=1}^k  \int_{t_{j-1}}^{t_j}  \!\!
\EE \Big( \!1_{\Omega_{j-1}^M}  \| G(u(s)) - G( u_N(t_{j-1}) )\|_{{\mathcal L}_2(K,H)}^2 \Big) ds .
\end{align*}

 The  Lipschitz condition {\bf (G1)}   
 and  \eqref{HolderV}      imply  for any $\epsilon >0$ and $s\in [t_{l-1}, t_l]$, 
\begin{align} 			\label{majG-G}
\EE\big( 1_{\Omega_{l-1}^M} & \| G(u(s)) - G(u_N(t_{l-1})) \|_{{\mathcal L}_2(K,H)}^2  \big) \nonumber \\
\leq & \; 
(1+\epsilon) \EE \big( 1_{\Omega_{l-1}^M}  \| G(u(t_{l-1})) - G( u_N(t_{l-1}))\|_{{\mathcal L}_2(K,H)}^2 \big)  \nonumber \\
&\;  + \big( 1+\frac{1}{\epsilon}\big) 
\EE\big( 1_{\Omega_{l-1}^M}  \|G( u(s)) - G(u(t_{l-1}))\|_{{\mathcal L}_2(K,H)}^2 \big) \nonumber \\
\leq  &   \;   L_1\, (1+\epsilon) \;  \EE \Big( 1_{\Omega_{l-1}^M} |e_{l-1}|_{\LL^2}^2\Big) +C(\epsilon)   \Big( \frac{T}{N} \Big)^{\eta}. 
\end{align} 
Since $\Omega_j^M \subset \Omega_{j-1}^M$ and $e_0=0$, we deduce  for any $k=2, \cdots, N,$
\begin{align}				\label{Maj_M2} 
 \EE\Big(  \!\max_{1\leq j\leq k} M_j(2) \Big) \leq   &
\tilde{ \delta}_2  \frac{T}{N}  \sum_{j=1}^k    \EE\big(1_{\Omega_{j-1}^M} \, |e_j- e_{j-1}|_{\LL^2}^2 \big) \!+ \!
 \frac{1+\epsilon}{4\tilde{\delta}_2}  L_1\,  \frac{T}{N}\ \sum_{j=1}^{k-1}  \EE \Big( 1_{\Omega_{j-1}^M} |e_{j}|_{\LL^2}^2\Big)  \nonumber \\
 & 
+ C(\epsilon,\tilde{\delta}_2)  T \Big( \frac{T}{N} \Big)^{\eta}. 
\end{align} 
Since $1_{\Omega_{l-1}^M}$ and $e_{l-1}$ are ${\mathcal F}_{t_{l-1}}$-measurable, 
using the Burkholder-Davies-Gundy inequality,  the Young inequality, \eqref{majG-G},  and using once more the inclusion 
$\Omega_j^M \subset \Omega_{j-1}^M$,   we deduce that for any $\tilde{\delta}_1 >0$,
\begin{align}				\label{Maj_M1} 
 \EE\Big(&  \max_{1\leq j\leq k} M_j(1) \Big) \leq 3 \sum_{l=1}^k 
 \EE\Big[ \Big\{ 1_{\Omega_{l-1}^M}  \int_{t_{l-1}}^{t_l} \|G(u(s)) - G( u_N(t_{l-1})) \|_{{\mathcal L}_2(K,H)}^2
\, |e_{l-1}|_{\LL^2}^2   ds  \Big\}^{\frac{1}{2}} \Big]  \nonumber \\
& \leq 3 \EE\Big[ \Big( \max_{1\leq l\leq k} 1_{\Omega_{l-1}^M} |e_{l-1}|_{\LL^2} \Big) \Big\{ \sum_{l=1}^k 1_{\Omega_{l-1}^M} \int_{t_{l-1}}^{t_l}
 \|G(u(s)) - G( u_N(t_{l-1}))\|_{{\mathcal L}_2(K,H)}^2  ds \Big\}^{\frac{1}{2}}\Big]   \nonumber \\ 
& \leq \tilde{\delta}_1 \EE \Big( \max_{1\leq l\leq k} 1_{\Omega_{l-1}^M} |e_{l-1}|_{\LL^2}^2 \Big) + \frac{9(1+\epsilon)}{4\tilde{\delta}_1} 
L_1\;  \frac{T}{N} \sum_{j=1}^{k-1}   \EE \Big( 1_{\Omega_{j-1}^M} |e_{j}|_{\LL^2}^2\Big) 
+ C(\epsilon,\tilde{\delta}_1)   \Big( \frac{T}{N} \Big)^{\eta} .  
\end{align}
Collecting the upper estimates \eqref{maj_max}--\eqref{Maj_M1} and  taking $\tilde{\delta}_2=1$, we deduce  for $k=2, \cdots, N$ 
\begin{align}					\label{maj_preGronwall}
 \EE\Big( &\max_{1\leq j\leq k} \Big\{ \frac{1}{2} 1_{\Omega_{j-1}^M} |e_j|_{\LL^2}^2 \Big) 
 + \nu \Big( 1-\sum_{i=0}^3 \delta_i\Big)  \frac{T}{N} \sum_{l=1}^j 
 1_{\Omega_{l-1}^M} |\nabla e_l|_{\LL^2}^2 \Big\} \Big)  \nonumber \\
& \leq   \Big[ \tilde{\delta}_1 +  \gamma_2 + (1+\epsilon)  \frac{\bar{C}^2 M}{4\delta_1 \nu}  \frac{T}{N} \Big]   
 \EE\Big( \max_{1\leq j\leq k} 1_{\Omega_{j-1}^M} |e_j|_{\LL^2}^2 \Big) \nonumber \\
 &\quad +   \Big[  (1+\epsilon)   \frac{\bar{C}^2 M}{4\delta_1 \nu}  +  \gamma_2 
 + \frac{1+\epsilon}{4}  \Big( \frac{1}{\tilde{\delta}_2} +\frac{9}{\tilde{\delta}_1} \Big) 
 L_1 \Big] \frac{T}{N}  \sum_{j=1}^{k-1} \EE\big(1_{\Omega_{j-1}^M} | e_j|_{\LL^2}^2 \big) + C\Big( \frac{T}{N}\Big)^\eta.  
\end{align}
\smallskip

\noindent {\bf Step 3: Discrete Gronwall lemma}   \\
Fix $\alpha \in (0,1) $ and choose $\delta_1 \in (0, 1-\alpha)$. Then 
\[  \nu(1-\delta_1) 
\geq \alpha \nu.\] 
 Fix  $\tilde{\epsilon}\in (0,1)$ and  $\gamma_2 \in \big( 0, \frac{\tilde{\epsilon}}{2} (\frac{1}{2}-\tilde{\delta}_1) )$;  
 suppose that  $N$ is 
large enough to imply 
\[ \frac{1}{2} - \tilde{\delta}_1 - \Big( \gamma_2 + (1+\epsilon)  \frac{\bar{C}^2 M}{4\delta_1 \nu} \Big)  \frac{T}{N}   \geq 
 (1-\tilde{\epsilon}) \big( \frac{1}{2}-\tilde{\delta}_1\big), \]
 and choose $\delta_i$, $i=0,2,3$ such that $\delta_0+\delta_2+\delta_3<\frac{\alpha}{2} \nu$.  Then for $N$ large enough, 
 
 \begin{align*}
 (1-&\tilde{\epsilon}) \big( \frac{1}{2} - \tilde{\delta}_1 \big) \EE\Big( \max_{1\leq j\leq k} 1_{\Omega_{j-1}^M} |e_j|_{\LL^2}^2 \Big) 
 + \frac{\alpha \nu}{2} \frac{T}{N} \sum_{j=1}^k 
\EE\big( 1_{\Omega_{j-1}^M} |\nabla e_j|_{\LL^2}^2\big) \nonumber \\
& \leq \Big[    (1+\epsilon)  \frac{\bar{C}^2 M}{4\delta_1 \nu}  +  
C(\gamma_2, \tilde{\delta}_1, 
\epsilon) L_1   \Big] \frac{T}{N}  \sum_{j=1}^{k-1} \EE\big(1_{\Omega_{j-1}^M} | e_j|_{\LL^2}^2 \big)  + C\Big(\frac{T}{N}\Big)^\eta. 
 \end{align*}
 Set 
\begin{equation} 		\label{C_1}
 C_1(M):= 
 \frac{    \frac{(1+\epsilon) \bar{C}^2 }{4\delta_1 \nu}  M 
 + (1+\epsilon) C( \gamma_2, \tilde{\delta}_1)}{(1-\tilde{\epsilon} ) \big(\frac{1}{2} - \tilde{\delta}_1\big)}.
\end{equation}  
Neglecting the second term on the left  hand side and using a discrete version of the Gronwall lemma, we deduce that
\[ \EE\Big( \max_{1\leq j\leq k} 1_{\Omega_{j-1}^M} |e_j|_{\LL^2}^2 \Big)  \leq C\,  e^{C_1(M)\, T} \, \Big(\frac{T}{N}\Big)^\eta.\]

Once these inequalities hold, choosing  $\epsilon \sim 0$, 
$\tilde{\delta}_1 \sim 0$, $\gamma_2\sim 0$, $\delta_1\sim 1$,
$\delta_i\sim 0$ for $i=0,2,3$, we may take $C_1(M)$ such that 
\[ 
C_1 (M)= \frac{ (1+\bar{\epsilon}) \, \bar{C}^2M}{  2   \nu} + C(\bar{\epsilon}) L_1, 
\] 
where $\bar{C}$ is the constant defined in \eqref{interpol} and $\bar{\epsilon}>0$ is arbitrary close to 0. 
  Indeed, given $\bar{\epsilon} >0$, we choose the other constants so that  
  $\frac{1+{\epsilon}}{\delta_1 (2-4\tilde{\delta}_1)}\leq \frac{1+\bar{\epsilon}}{2}$. 
Plugging this in the previous upper estimate and using the inclusions $\Omega_k^M\subset \Omega_j^M$ for $j=1, \cdots, k$, we
deduce  \eqref{loc_moments} and \eqref{def_C1_Euler}. 
 
\end{proof} 


\subsection{ Strong speed of convergence of the implicit Euler scheme}	\label{strong_Euler}
As in section \ref{sec_strong_splitting}, let us use the H\"older inequality with conjugate exponents   $2^{q-1}$ and $p=\frac{2^{q-1}}{2^{q-1}-1}$.
  We obtain 
\begin{align}		\label{Holder_L2}			
\EE\Big( 1_{(\Omega_N^M)^c} \max_{1\leq k\leq N} |e_k|_{\LL^2}^2  \Big) \leq &\,  C \, 
\Big[ \PP\big( (\Omega_N^M)^c\big) \Big]^{\frac{1}{p}} \nonumber \\
& \quad \times \Big[
\EE \Big( \sup_{0\leq s\leq T} |u(s)|_{\LL^2}^{2^q} + \max_{0\leq k\leq N} |u_N(t_k)|_{\LL^2}^{2^q} \Big) \Big]^{ \frac{1}{2^{q-1}}  }, \\
\EE\Big( 1_{(\Omega_N^M)^c} \frac{T}{N} \sum_{k=1}^N  | \nabla e_k|_{\LL^2}^2  \Big) \leq & \, 
C \Big[ \PP\big( (\Omega_N^M)^c\big) \Big]^{\frac{1}{p}}\nonumber \\
&\quad  \times \Big[
\EE \Big( \sup_{0\leq s\leq T} |\nabla u(s)|_{\LL^2}^{2^q} + \max_{0\leq k\leq N} |\nabla u_N(t_k)|_{\LL^2}^{2^q} \Big) 
\Big]^{  \frac{1}{2^{q-1}}  }. 
\label{Holder_nabla}
\end{align} 

\noindent The inequalities  \eqref{bound_u} and \eqref{moments_scheme}
 prove that if  $\EE(\|u_0\|_V^{2^q})<\infty$,
the second factors on  the right hand sides of \eqref{Holder_L2} and \eqref{Holder_nabla} 
are bounded by a constant independent of $N$. 

\noindent We now upper estimate the probability of the complement of the localization set and to balance the upper estimates of the $L^2$
moments localized on the set $\Omega_N^M$ and its complement. 
To obtain a strong speed of convergence will require the threshold $M$ to depend on $N$.
Two cases are studied.
\smallskip

\noindent{\bf Case 1: Linear growth diffusion coefficient}

\noindent  Suppose that $G$ satisfies conditions {\bf (G1)} and {\bf (G2)} 
and that $\EE(\|u_0\|^{2^q})<\infty$. 
Then \eqref{bound_u} implies
 \begin{align}			\label{upper_lingrowth}
\PP\Big( \big( \Omega_N^{M(N)}\big)^c\Big) \; \leq & \; \PP\Big( \sup_{0\leq s\leq T} |\nabla u(s)|_{\LL^2}^2 > M(n) \Big) \nonumber  \\
\leq &\; \Big(  \frac{1}{M(N)} \Big)^{2^{q-1}} \; \EE\Big( \sup_{0\leq s\leq T} \|u(s)\|_V^{2^q} \Big) \leq C_q \, M(N)^{-2^{q-1}}. 
\end{align}
If we suppose that $\EE(\|u_0\|_V^{2^q}) <\infty$, in order to balance the upper estimates \eqref{Holder_L2}, \eqref{Holder_nabla} with 
\eqref{upper_lingrowth} and \eqref{loc_moments}, we have to choose $M(N)\to \infty$ as $N\to \infty$,  such that as $N\to \infty$, 
\[ 
\Big( \frac{T}{N} \Big)^\eta \exp[C_1\big(M(N)\big)\, T ] \; \asymp \;  C(q) M(N)^{- 2^{q-1} +1 } .
\] 
where $C_1(M(N))$ is defined in \eqref{def_C1_Euler}. 
 Fix $\bar{\epsilon}>0$; taking logarithms and neglecting constants leads to   
\[  -\eta \ln(N) + \frac{ (1+\bar{\epsilon})  \bar{C}^2 M(N) T}{2  \nu} 
\asymp - \big( 2^{q-1}  -  1 ) \ln(M(N)) +  O(1)  \quad \mbox{\rm as } N\to \infty. \]
Let 
\begin{equation} \label{M(N)-Euler-lingro}
M(N):= \frac{ 2 \nu }{(1+\bar{\epsilon})\bar{C}^2 T} \Big\{ \eta \ln(N) 
-     \big(2^{q-1}  -  1 \big)  \ln\big(  \ln(N) \big) \Big\}
 \asymp \frac{2 \nu \eta \ln(N)}{ (1+\bar{\epsilon}) \bar{C}^2 T} .
\end{equation}

Then,  for this choice of $M(N)$, we have 
\[ -\eta \ln(N) + C_1\big( M(N) \big) T =   -   \ln\big[ \big( \ln(N)\big)^{2^{q-1}- 1}\big]  + O(1), \]
 which implies $\big( \frac{T}{N}\big)^\eta \exp[ C_1(M(N)) T] \asymp C \big( \ln(N)\big)^{-2^{q-1}+1}$ for some positive  constant $C$. 
Furthermore, $M(N)^{-2^{q-1}+1} \asymp C \big( \ln(N)\big)^{-2^{q-1}+1}$ for some positive constant $C$. 
Similar computations with 
 sum of the $V$ norms of the error on the time grid yield 
\[ 
\EE\Big( \max_{1\leq k\leq N} |e_k|_{\LL^2}^2 +\nu \frac{T}{N} \sum_{k=1}^N |\nabla e_k|_{\LL^2}^2 \Big) \leq 
 C\big( \ln(N)\big)^{-(2^{q-1} - 1)}  ,   
\] 
for some constant $C$ depending on $T$, $q$ and the coefficients $K_i$, $i=0,1$.  
This completes the proof of the following
\begin{theorem}		\label{th_Euler_lingrow}
Let $u_0$ be such that $\EE(\|u_0\|_V^{2^q})<\infty$ for some $q\geq 3$, $G$ satisfy assumptions  {\bf(G1)} and {\bf (G2)}. 
Then the fully implicit scheme  
$u_N$ solution of \eqref{full-imp1}  converges in $L^2(\Omega)$ to the solution $u$ of \eqref{2D-NS}. More precisely, for $N$ large enough we have
\begin{align}			\label{speed_lin_growth_Euler}
\EE\Big( \max_{1\leq k \leq N} |u(t_k)-u_N(t_k)|_{\LL^2}^2  + \frac{T}{N} \sum_{k=1}^N \big|\nabla\big[ u(t_k)-u_N(t_k)\big] \big|_{\LL^2}^2 \Big) 
\leq    C \big[\ln(N)\big]^{-(2^{q-1} - 1)}.    
\end{align}
\end{theorem}
\begin{remark}
Note that,  as for the splitting scheme, 
 if $u_0$ is a deterministic element of $V$ and $G$ satisfies the conditions {\bf (G1)} and {\bf (G2)}, 
  we have
\[ \EE\Big( \max_{1\leq k \leq N} |u(t_k)-u_N(t_k)|_{\LL^2}^2  + \frac{T}{N} \sum_{k=1}^N \big|\nabla\big[ u(t_k)-u_N(t_k)\big] \big|_{\LL^2}^2 \Big) 
\leq  C \big[\ln(N)\big]^{-\gamma}\]
 for any $\gamma >0$. This upper estimate is also true if $\|u_0\|_V$ has moments of all orders, for example if $u_0$ is a $V$-valued
 Gaussian random variable independent of the noise $W$.
\end{remark}

\smallskip 

\noindent{\bf Case 2: Additive noise} Suppose that $G(u):=G \in {\mathcal L}_2(K,V)$, that is the noise is additive, or more generally
that the conditions {\bf (G1)} and {\bf (G2)} are satisfied with  $K_1=0$ 
Using an exponential Markov inequality, we deduce that for any constant $\alpha >0$
\begin{align}				\label{upper_Exp}
\PP\Big( \big( \Omega_N^{M(N)}\big)^c\Big) \; \leq & \exp\big(-\alpha M(N)\big) \; 
 \EE \Big[ \exp\Big(\alpha \sup_{0\leq t\leq T} |\nabla u(t)|_{\LL^2}^2\Big) \Big].
\end{align}
Recall that Lemma \ref{exp_mom} implies that for $\alpha \in (0, \alpha_0]$, where $\alpha_0=\frac{\nu}{4 K_0 \tilde{C}}$ and $\tilde{C}$
is  defined in \eqref{Poincare}, we have $\EE\big[ \sup_{t\in [0,T]} \exp(\alpha \|u(t)\|_V^2) \big]<\infty$. 
Using \eqref{loc_moments} with \eqref{def_C1_Euler}, \eqref{upper_Exp} and \eqref{moments_exp},  
we choose $M(N)$ such that 
\begin{equation} 			\label{ln_add_Euler}
  \Big(\frac{T}{N}\Big)^\eta \exp\Big( \frac{ (1+\bar{\epsilon})   \; \bar{C}^2 M(N)T}{  2 } \nu \Big) 
  = c_2\exp\Big(- \frac{\nu  M(N) }{p 4 K_0  \tilde{C} }  \Big),
\end{equation} 
for some $p>1$, $\bar{\epsilon}>0$, and  some positive constant $c_2$,
 where $\bar{C}$ (resp. $\tilde{C}$) is defined by \eqref{interpol} (resp. \eqref{Poincare}). 
  For any $p \in (1,\infty)$  since $u_0$ is deterministic, 
$\EE(\|u_0\|_V^q)<\infty$ for conjugate exponents $p$ and $q$.  Set
 \[ M(N):= \frac{\eta \ln(N)}{\frac{\nu}{p4K_0 \tilde{C} }+ \frac{(1+\bar{\epsilon}) \bar{C}^2 T}{2\nu}}\] 
 for some $\bar{\epsilon}>0$.  
Then  $M(N)\to \infty$ as $n\to \infty$, and both hand sides of \eqref{ln_add_Euler} 
are equal to some constant multiple of $N^{-\beta \eta}$, where, choosing $p$ close enough
to 1 and $\bar{\epsilon} \sim 0$,  we have 
$\beta <\frac{\frac{\nu}{4K_0\tilde{C}}}{\frac{\nu}{4K_0\tilde{C}} + \frac{\bar{C}^2 T}{ 2  \nu}}$.  
 Since $\eta<\frac{1}{2}$ can be chosen as close to $\frac{1}{2}$ as wanted, this yields the following rate of convergence. 
\begin{theorem}  \label{th_Euler_exp}
Let $u_0\in V$, $G$ satisfy assumptions {\bf (G1)} and {\bf (G2)} with $K_1=L_1=0$. 
Let $u$ denote the solution of
\eqref{2D-NS} and $u_N$ be the fully implicit scheme solution of \eqref{full-imp1}. Then for $N$ large enough, 
$\bar{C}$ (resp.  $\tilde{C}$) defined by \eqref{interpol} (resp. \eqref{Poincare}), 
\begin{equation}			\label{speed_add_Euler}
\EE\Big( \max_{1\leq k\leq N} |u(t_k)-u_N(t_k)|_{\LL^2}^2 + \frac{T}{N} \sum_{k=1}^N \big| \nabla \big[ u(t_k)-u_N(t_k) \big] \big|_{\LL^2}^2\Big)
\leq  C \Big( \frac{T}{N}\Big)^{\gamma}, 
\end{equation}
where  $\gamma <\frac{1}{2} \left(\frac{\frac{\nu}{4 K_0\tilde{C}}}{\frac{\nu}{4K_0\tilde{C}} + \frac{\bar{C}^2 T}{2\nu}}\right)$. 
\end{theorem} 

Note that if $\nu$ is large, the speed of convergence of the $H$ and $V$ norms in Theorem \ref{th_Euler_exp} is
  "close" to $C(T) N^{-\frac{1}{2}}$.   
Intuitively, it cannot be better because of the stochastic integral
and the scaling between the time and space parameters in the heat kernel, which is behind the time regularity of the solution stated in 
\eqref{HolderV}.  
\bigskip

 \subsection{Semi-implicit Euler scheme} In this section, we prove the strong $L^2(\Omega)$ convergence of a discretization scheme with
 a linearized drift. Let $v_N$ be defined on the time grid $(t_k, k=0, \cdots, N)$ as follows.\\
 {\bf Semi implicit Euler scheme} Let $u_0$ be a $V$-valued  ${\mathcal F}_0$-measurable random variable and set  $v_N(0)=u_0$. 
  For $k=1, \cdots,
 N$, let $v_N(t)\in V$ be such that $\PP$ a.s. for all $\phi \in V$,
 \begin{align} 			\label{semi-imp}
 \big( v_N(t_k)-v_N(t_{k-1})\, , \, \phi\big) + \frac{T}{N} \Big[ \nu \big( \nabla v_N(t_k)\, , &\, \nabla \phi \big) + \big\langle B(v_N(t_{k-1}) , v_N(t_k))
 \, , \, \phi \big\rangle  \nonumber \\
 &=\big( G(v_N(t_{k-1}) ) \, \Delta_k W\, , \, \phi\big),
 \end{align}
 where $\Delta_k W = W(t_k) - W(t_{k-1})$. 
 
 Note that since in general $\langle B(u,v)\, , \, Av\rangle \neq 0$ for $u,v\in \mbox{\rm Dom}(A)$, the moments of $v_N$ 
 are bounded in a weaker norm than
 that of the fully implicit scheme $u_N$. 
 \begin{lemma}			\label{lem_mom_semi} Let $u_0\in L^{2^q}(\Omega,V)$ for   some integer $q\geq 2$ 
 be ${\mathcal F}_0$-measurable and let $G$ satisfy the condition {\bf(G1)}.   
 Then each random variable $v_N(t_k)$, $k=0, \cdots, N$ is
 ${\mathcal F}_{t_k}$-measurable such that 
 \begin{align}				\label{moments_semi}
 \sup_N \EE\Big( \max_{1\leq k\leq N} |v_N(t_k)|_{\LL^2}^{2^q} + 
 \nu \frac{T}{N} \sum_{k=1}^N |v_N(t_k)|_{\LL^2}^{2^{q-1}} \, \|v_N(t_k)\|_V^2 \Big) \leq C(T,q).
 \end{align}
 \end{lemma}
 For $k=0, \cdots, N$, set $\bar{e}_k=u(t_k)-v_N(t_k)$. Unlike \cite{CarPro}, we will not compare the schemes $u_N$ and $v_N$ since the
 norm of the difference would require a localization in terms of the gradient of $u_N$.
  Instead of that, we prove the following analog of Proposition \ref{Loc_cv_Euler}.
 \begin{prop}			\label{Loc_semi}
 Let $G$ satisfy the  growth and Lipschitz conditions {\bf (G1)} and {\bf (G2)},  
and $u_0$ be ${\mathcal F}_0$-measurable such that $\EE(\|u_0\|_V^8)<\infty$. 
 Then for $\Omega^M_k$ defined by \eqref{Omegak} and $N$ large enough, we have for $k=1, \cdots, N$ and  $\eta <\frac{1}{2}$
 \begin{align}		\label{loc_semi}
 \EE\Big( 1_{\Omega^M_{k-1}} \Big[ \max_{1\leq j\leq k}  |\bar{e}_j|_{\LL^2}^2 
 + \nu \frac{T}{N} \sum_{j=1}^k |\nabla \bar{e}_j|_{\LL^2}^2\Big] 
 \Big) \leq C \Big( \frac{T}{N}\Big)^\eta \exp\big[ C_1(M) T \big],
 \end{align}
 where $C>0$ is some constant and $C_1(M)$ is defined by \eqref{def_C1_Euler}  for any $\bar{\epsilon}>0$. 
 \end{prop}
 \begin{proof} Many parts of the argument are similar to the corresponding ones in the proof of Proposition \ref{Loc_cv_Euler}; we only focus
 on the differences. 
 
 We first consider the duality between the difference of bilinear terms and $\bar{e}_j$, that is upper estimate 
 $\int_{t_{j-1}}^{t_j} \langle B(u(s),u(s)) - B(v_N(t_{j-1}), v_N(t_j))\, , \, \bar{e}_j \rangle ds$. For every $s\in [t_{j-1}, t_j]$, using the bilinearity
 and antisymetry of $B$ we deduce
 \[ \langle B(u(s),u(s)) - B(v_N(t_{j-1}), v_N(t_j))\, , \, \bar{e}_j \rangle = \sum_{i=1}^3 \bar{T}_i(s), \]
 where
 \begin{align*}
 \bar{T}_1(s):=& \big\langle B\big(\bar{e}_{j-1} , v_N(t_j) \big) \, , \, \bar{e}_j \big\rangle 
 = \big\langle B\big(\bar{e}_{j-1} , u(t_j) \big) \, , \, \bar{e}_j \big\rangle, \\
 \bar{T}_2(s):=&  \big\langle B\big(u(s)-u(t_{j-1})  , u(t_j) \big) \, , \, \bar{e}_j \big\rangle, \\
 \bar{T}_3(s):= & \big\langle B\big(u(s), u(s)-  u(t_j) \big) \, , \, \bar{e}_j \big\rangle = -\big\langle B(u(s), \bar{e}_j)\, , \, u(s)-u(t_j)\big\rangle.
 \end{align*}
 Using the H\"older inequality, \eqref{interpol} and  the Young inequality,  we deduce that for every $\delta_1>0$, 
 \begin{align}		\label{maj_Tbar1}
 \int_{t_{j-1}}^{t_j} |\bar{T}_1(s)| ds \leq & \bar{C} \frac{T}{N} |\bar{e}_{j-1}|_{\LL^2}^{\frac{1}{2}}  |\bar{e}_{j}|_{\LL^2}^{\frac{1}{2}} 
 |\nabla \bar{e}_{j-1}|_{\LL^2}^{\frac{1}{2}}   |\nabla \bar{e}_{j}|_{\LL^2}^{\frac{1}{2}} |\nabla u(t_j)|_{\LL^2}. \nonumber \\
\leq  &\frac{\delta_1}{2} \nu \frac{T}{N}  |\nabla \bar{e}_{j-1}|_{\LL^2}^2 + \frac{\delta_1}{2} \nu \frac{T}{N}  |\nabla \bar{e}_{j}|_{\LL^2}^2\nonumber \\
&\;+ \frac{\delta_1\bar{C}^2}{8\delta_1 \nu}   \frac{T}{N}  | \bar{e}_{j-1}|_{\LL^2}^2 |\nabla u(t_j)|_{\LL^2}^2 +
+ \frac{\delta_1\bar{C}^2}{8 \delta_1 \nu}   \frac{T}{N}  | \bar{e}_{j}|_{\LL^2}^2 |\nabla u(t_j)|_{\LL^2}^2. 
 \end{align}
 The upper estimates of $\int_{t_{j-1}}^{t_j} \!  \bar{T}_i(s) ds$, $i=2,3$ are similar to the corresponding ones in the first step of the proof of 
 Theorem \ref{Loc_cv_Euler}. This yields the following analog of \eqref{(e-e,e)} with the same upper estimates \eqref{bound_ET1} --
 \eqref{bound_ET3} of the terms $\tilde{T}_j(i)$, 
 $i=1,2,3$
 \begin{align}				\label{(e-e,e)semi}
 (\bar{e}_j-\bar{e}_{j-1}\,& , \, \bar{e}_j)  + \nu |\nabla \bar{e}_j|_{\LL^2}^2 \leq \nu (\delta_0 + \frac{1}{2} \delta_1 + \delta_2+\delta_3) \frac{T}{N}
 |\nabla \bar{e}_j|_{\LL^2}^2 + \frac{1}{2} \delta_1 \nu \frac{T}{N} |\nabla \bar{e}_{j-1}|_{\LL^2}^2  \nonumber \\
 &+ \Big( \gamma_2 + \frac{\bar{C}^2}{8\delta_1 \nu} |\nabla u(t_j)|_{\LL^2}^2 \Big) \frac{T}{N} |\bar{e}_j|_{\LL^2}^2 +
 \frac{\bar{C}^2}{8\delta_1 \nu} |\nabla u(t_j)|_{\LL^2}^2 | \bar{e}_{j-1}|_{\LL^2}^2 
  \nonumber \\
  &+ \sum_{i=1}^3 \tilde{T}_j(i) +\Big( \bar{e}_j \, , \, \int_{t_{j-1}}^{t_j} \big[ G(u(s))-G(v_N(t_{j-1})) \big] \, dW(s) \big).  
 \end{align}
 Once adding these estimates localized on the set $\Omega^M_{t_{j-1}}$, we deduce an upper estimate similar to \eqref{maj_max}
 where $e_j$ is replaced by $\bar{e}_j$. Following the same steps as in the proof of Proposition \ref{Loc_cv_Euler}, we conclude the proof.
 \end{proof}

The arguments in section \ref{strong_Euler} prove that the statements of Theorems \ref{th_Euler_lingrow} and 
\ref{th_Euler_exp} remain valid if we replace the solution $u_N(t_k)$
of the fully implicit Euler scheme by the solution $v_N(t_k)$ of the semi implicit one.

\subsection{Time dependent coefficients}  \label{sec_time_dependent}
 For the sake of simplicity, we have supposed that the diffusion coefficient $G$ does not depend on time. An easy modification
of the proofs of this section shows that the statements of Theorems \ref{th_Euler_lingrow}, \ref{th_Euler_exp} for the fully or semi implicit
Euler schemes remain true if we suppose that $G:[0,T]\times V\to {\mathcal L}_2(K,H)$
(resp. $G:[0,T]\times {\rm Dom(A)}\to {\mathcal L}_2(K,V)$) satisfies  the following global linear growth and Lipschitz conditions 
similar to those imposed in assumptions  {\bf (G1)} and {\bf (G2)} 
\begin{align*}
&\|G(t,u)\|_{{\mathcal L}_2(K,H)}^2 \leq K_0 + K_1 |u|_{\LL^2}^2,  \quad 
\|G(t,u)\|_{{\mathcal L}_2(K,V)}^2 \leq K_0 + K_1 \|u\|_{V}^2 ,\\ 
&\|G(t,u)-G(t,v)\|_{{\mathcal L}_2(K,H)^2}^2 \leq L_1 |u-v|_{\LL^2}^2, \\
&\|G(t,u)-G(t,v)\|_{{\mathcal L}_2(K,V)^2}^2 \leq L_1 \|u-v\|_V^2 , 
\end{align*} 
for $u,v\in V$ (resp. $u,v\in \mbox{\rm Dom }(A)$).  
Furthermore,  the diffusion coefficient  $G$ should also satisfy the following  time regularity condition   \\
{\bf (G3)} There exits a constant $C>0$ such that for any $u,v\in V$ and $s,t\in [0,T]$:
\[ \|G(t,u)-G(s,u)\|_{{\mathcal L}_2(K,H)}^2 \leq C |t-s|^{\frac{1}{2}} \,  (1+\|u\|_{{\mathbb L}^2}^2).   \]
In that case, the fully implicit scheme $u_N$ 
 (resp. semi implicit scheme $v_N$) is defined replacing $G(u_N(t_{k-1}))$ by $G(t_{k-1}, u_N(t_{k-1}))$
(resp. by   $G(t_{k-1}, v_N(t_{k-1}))$)  on  the right hand side of  \eqref{full-imp1}.

\section{Appendix} \label{Appendix}
In this section we prove two technical lemmas used to obtain the strong convergence results. 

\subsection{Proof of Lemma \ref{supt_ZN}} \label{A1}

 First note that using \eqref{moments_splitting} and \eqref{zN-uN_V}, we deduce that  if $\EE(\|u_0\|_V^{2p})<\infty$ then 
\[  \sup_{N\geq 1} \EE\int_0^T \|z^N(s)\|_V^{2p} dt \leq C(p).\]
Thus only moments of $\sup_{t\in [0,T]}|z^N(t)|_{\LL^2}^{2p}$ have to be dealt with. 

The process $z^N$  defined by \eqref{def_zN} is not regular enough to apply directly It\^o's formula to $|z^N(t)|_{\LL^2}^2$. 
Hence, as in the proof of Proposition \ref{prop5.1-mod}, 
we need to apply  It\^o's formula on the smooth Galerkin 
approximations of the processes $u^N$, $y^N$ and $z^N$, and then pass to the limit.  This yields for every $t\in [0,T]$
\begin{align*} 		
|z^N(t)|_{\LL^2}^2 = |&u_0|_{\LL^2}^2 -2 \int_0^t  \langle F(u^N(s))\, , \, z^N(s)\rangle ds + \int_0^t \|G( y^N(s)\|_{{\mathcal L}_2(K,H)}^2 ds
\nonumber \\
& + 2\int_0^t  \big(  z^N(s) \, , \, G(y^N(s)) dW(s)  \big).  
\end{align*}
Using once more the It\^o formula, we deduce 
\begin{equation} 		\label{Ito_zN_p}
|z^N(t)|_{\LL^2}^{2p} = |u_0|_{\LL^2}^{2p} + I(t) + \sum_{i=1}^3 J_i(t),
\end{equation}
where
\begin{align*}
I(t)=& 2p \int_0^t  \big(  z^N(s) \, , \, G(y^N(s)) dW(s) \big) |z^N(s)|_{\LL^2}^{2(p-1)}, \\
J_1(t)=& -2p \nu  \int_0^t |z^N(s)|_{\LL^2}^{2(p-1)} \big( \nabla u^N(s) \, , \nabla z^N(s)\big) ds  \\
J_2(t)=& - 2p \int_0^t   |z^N(s)|_{\LL^2}^{2(p-1)} \big\langle B(u^N(s), u^N(s)) \, , \, z^N(s) \big\rangle ds \\
J_3(t)=& + p \int_0^t |z^N(s)|_{\LL^2}^{2(p-1)} \|G( y^N(s)\|_{{\mathcal L}_2(K,H)}^2 ds\\
&+ 2p(p-1) \int_0^t \|G^*( y^N(s)) z^N(s) \|_K^2  |z^N(s)|_{\LL^2}^{2(p-2)} ds.
\end{align*}
The H\"older and Young inequalities  imply
\begin{align*} 		
|J_1(t)| \leq & 2 (p-1)  \nu \int_0^t |z^N(s)|_{\LL^2}^{2p} ds + \nu  \int_0^t \big[ |\nabla z^N(s)|_{\LL^2}^{2p} + |\nabla u^N(s)|_{\LL^2}^{2p}\big] ds 
\end{align*}
Using again the H\"older and Young inequalities with exponents $\frac{2p+1}{2p-1}$ and $\frac{2p+1}{2}$, we deduce
\begin{align*} 
|J_2(t)| \leq & 2p \int_0^t |z^N(s)|_{\LL^2}^{2(p-1)} |\nabla z^N(s)|_{\LL^2} \; \|u^N(s)\|_X^2 ds \\
\leq &  
 \frac{(2p-1) 2p}{2p+1}   \int_0^t \|z^N(s)\|_V^{2p+1} ds
 + \frac{4p}{2p+1} \Big( \frac{\bar{C}}{2} \Big)^{\frac{2p+1}{2}} \int_0^t \|u^N(s)\|_V^{2p+1} ds,
\end{align*}
where $\bar{C}$  is the constant defined in \eqref{interpol}. 
Finally, using the growth condition  {\bf (G1)}, 
we deduce
\begin{align*}
|J_3(t)| \leq & (2p^2-p) \int_0^t |z^N(s)|_{\LL^2}^{2(p-1)} \big[ K_0 + K_1 |y^N(s)|_{\LL^2}^2 \big] ds \\
\leq & 
 (2p-1)(p-1)    \int_0^t |z^N(s)|_{\LL^2}^{2p} ds +  C(p)  K_1^p   \int_0^t \|y^N(s)\|_{\LL^2}^{2p} ds + C(p) K_0 T .
\end{align*}
where  $C(p)$ is a constant depending on $p$.   The inequalities 
\eqref{moments_splitting} and \eqref{zN-uN_V},  and the above estimates of $J_i(t)$ for $i=1,2,3,$ imply  
the existence of a positive constant $C(p)$ depending on $p$
such that  for 
every integer. $N\geq 1$, 
\begin{equation}			\label{maj_J_zN}
\sum_{i=1}^3 \EE\Big( \sup_{t\in [0,T]}  |J_i(t)| \Big)  \leq C(p).
\end{equation}  
Furthermore, the Burkholder-Davies-Gundy inequality, the growth condition in  {\bf (G1)} 
and the Young inequality 
 imply 
\begin{align}    \label{maj_I_zN}
\EE\Big( \sup_{t\in [0,T]} |I(s)|\Big) \leq & 6p  \EE\Big( \Big\{  \int_0^T |z^N(s)|_{\LL^2}^{4p-2} 
 \|G(y^N(s))\|_{{\mathcal L}_2(K,H)}^2 ds  \Big\}^{\frac{1}{2}}
 \Big)  \nonumber  \\
& \leq  \frac{1}{2} \EE\Big( \sup_{t\in [0,T]} |z^N(t)|_{\LL^2}^{2p}\Big) + C(p)  
\EE  \int_0^T \big[ K_0^p + K_1^p |y^N(s)|_{\LL^2}^{2p}\big]  ds  \nonumber \\
& \leq \frac{1}{2} \EE\Big( \sup_{t\in [0,T]} |z^N(t)|_{\LL^2}^{2p}\Big) + C(p) ,
\end{align}
where the last inequality is deduced from \eqref{moments_splitting}.

The upper estimates \eqref{Ito_zN_p}--\eqref{maj_I_zN}, \eqref{bound_u}, \eqref{moments_splitting} and \eqref{zN-uN_V}  conclude the proof. 
\hfill $\square$

\subsection{Proof of Lemma \ref{exp_mom}} \label{A2} 
We prove 
 the existence of exponential moments for the square of the $V$ norm of the solution $u$
of \eqref{2D-NS}. \\
{\it Proof of Lemma \ref{exp_mom}.} 
Let $\alpha >0$;  we apply the  It\^o formula to the square of the $V$ norm of the process $u$ solution of \eqref{2D-NS}.
As explained in the proofs of Proposition \ref{prop5.1-mod} and of the previous Lemma, the It\^o formula 
has to be performed on smooth processes, for example a Galerkin or a  Yosida approximation of $u$,  
and then pass to the limit or use computations similar to those in \cite{ChuMil}, step 4 of the Appendix on page 416.   This yields
\begin{align}				\label{Ito_V_u}
\alpha \|u(t)|_V^2 &+ \alpha \nu \int_0^t |||u(s)|||^2 ds = \alpha \|u_0\|_V^2 + \alpha \int_0^t \|G(s,u(s)) \|_{{\mathcal L}_2(K,V)}^2 ds \nonumber \\
& + 2\alpha \int_0^t \big( u(s) \, , \, G(s,u(s)) dW(s)\big)_V - \alpha \nu \int_0^t |||u(s)|||^2 ds,
\end{align} 
where for $u,v\in V$ we set $(u,v)_V=(u,v) + (\nabla u, \nabla v)$ and recall that $|||u|||^2 := |\nabla u|_{\LL^2}^2 + | Au|_{\LL^2}^2$.  

Let $M(t):=2\alpha \int_0^t \big( u(s) \, , \, G(s,u(s)) dW(s)\big)_V$; then $M$ is a martingale with quadratic variation 
\begin{align*}
 \langle M\rangle_t &\leq 4\alpha^2 \! \int_0^t \! \|u(s)\|_V^2 \|G(s,u(s))\|_{{\mathcal L}_2(K,V)}^2 ds \leq 4 \alpha^2 K_0 \! \int_0^t  \!\|u(s)\|_V^2  ds,\\
 &\leq 4 \alpha^2 K_0 \tilde{C}\! \int_0^t  \! |||u(s)|||^2  ds,
\end{align*}
where the last inequality follows from \eqref{growthG_V} and the definition of $\tilde{C}$ in \eqref{Poincare}. 

Let  
 $\alpha_0:= \frac{\nu}{4 K_0 \tilde{C}}$;   then  
we deduce 
\[ M_t - \alpha \nu \int_0^t |||u(s)|||^2 ds \leq M_t - \frac{\alpha_0}{\alpha} \langle M\rangle_t. \]
Therefore, using the previous inequality and classical exponential martingale arguments, we deduce 
\begin{align*}
\PP \Big[ \sup_{0\leq t\leq T} & \Big(  M_t - \alpha \nu \int_0^t |||u(s)|||^2 ds   \Big) \geq K \Big] \leq 
\PP \Big[ \sup_{0\leq t\leq T}\big(  M_t - \frac{\alpha_0}{\alpha} \langle M \rangle_t \big) \geq K \Big] 
\\
 \leq &\PP\Big[  \sup_{0\leq t\leq T} \exp\Big( \frac{2\alpha_0}{\alpha} M_t 
 - \frac{1}{2} \big\langle \frac{2\alpha_0 }{\alpha} M  \big\rangle_t \Big)   \geq \exp\Big(\frac{2\alpha_0}{\alpha}K\Big) \Big] \\
 \leq & \exp\Big( - \frac{2\alpha_0}{\alpha}K\Big)\; \EE\Big[ \exp\Big( \frac{2\alpha_0}{\alpha} M_T 
 - \frac{1}{2} \langle \frac{\beta}{\alpha} M\rangle_T \Big) \Big]  \leq  \exp\Big( - \frac{2\alpha_0 }{\alpha}K\Big)
\end{align*}
for any $K>0$. Set 
\[ X=\exp\Big( \alpha\Big\{ \sup_{t\in [0,T]} 2\int_0^t \big( u(s)\, , \, G(s,u(s)) \big)_V
- \nu \int_0^t |||u(s)|||^2 ds\Big\} \Big); \]
then for $\alpha \in (0, \alpha_0]$, we deduce that $\PP(X\geq e^K) \leq \exp(-K \frac{2\alpha_0}{\alpha}) \leq e^{-2K}=(e^{-K})^2$ 
for any $K > 0$.\\
 Using this inequality
for any $C=e^K>1$ with $K>0$, we deduce 
$\EE(X) \leq 2+\int_0^\infty \PP(X\geq C) dC \leq 3$. Since   \eqref{Ito_V_u} implies
\begin{align*}
\alpha \Big( \sup_{t\in [0,T]}& \Big[ \|u(t)\|_V^2 + \nu \int_0^t |||u(s)|||^2 ds \Big] \Big) \leq \alpha \|u_0\|_V^2 + \alpha K_0 T \\
& + \sup_{t\in [0,T]}  \Big( \alpha\Big\{  2\int_0^t \big( u(s)\, , \, G(s,u(s)) \big)_V
- \nu \int_0^t |||u(s)|||^2 ds\Big\} \Big), 
\end{align*}
we conclude the proof of \eqref{moments_exp}. \hfill $\square$
\bigskip

\noindent {\bf Acknowledgements.} This   research was  started
in the fall 2017 while Annie Millet visited the University of Wyoming. 
She would like to thank this University  for the hospitality and the very pleasant  working conditions.
 Hakima Bessaih is partially supported by NSF grant DMS-1418838.\\
 Finally we would like to thank the anonymous referee for the careful reading and valuable remarks which helped improving this paper.

\end{document}